\documentclass{amsart}
\usepackage{epsfig,psfrag,amssymb,latexsym}

\newtheorem{theorem}{Theorem}[section]
\newtheorem{lemma}[theorem]{Lemma}
\newtheorem{proposition}[theorem]{Proposition}
\newtheorem{corollary}[theorem]{Corollary}

\theoremstyle{definition}

\theoremstyle{remark}
\newtheorem{remark}[theorem]{Remark}

\numberwithin{equation}{section}

\begin{document}

\title{From subfactor planar algebras to subfactors}

\author{Vijay Kodiyalam}
\address{The Institute of Mathematical Sciences, Taramani, Chennai 600113, India}
\email{vijay@imsc.res.in}
\author{V. S. Sunder}
\address{The Institute of Mathematical Sciences, Taramani, Chennai 600113, India}
\email{sunder@imsc.res.in}

\begin{abstract}
We present a purely planar algebraic proof of the main result of
a paper of Guionnet-Jones-Shlaykhtenko which constructs an extremal
subfactor from a subfactor planar algebra whose standard invariant is
given by that planar algebra.
\end{abstract}

\maketitle

\section{Introduction}

This paper contains no new results. What it does is to give a different
proof of the main result of \cite{GnnJnsShl} which, in turn, offers an
alternative proof of an important result of Popa in \cite{Ppa} that may be paraphrased as saying that any subfactor planar algebra arises
from an extremal subfactor.

In this introductory section, we will briefly review the main result of \cite{GnnJnsShl} and the ingredients of its proof and compare and contrast
the proof presented here with that one.

The main result of \cite{GnnJnsShl} begins with the construction of
a tower $Gr_k(P)$ of graded $*$-algebras with compatible traces $Tr_k$ associated to a subfactor planar algebra $P$. 
Next, appealing to a result in \cite{PpaShl}, the planar algebra $P$
is considered as embedded as a planar subalgebra of the planar
algebra of a bipartite graph as described in \cite{Jns2}.
It is shown that that the
traces $Tr_k$ are faithful and positive and then $II_1$-factors $M_k$ are obtained
by appropriate completions (in case the modulus $\delta >1$ - which is the only real case of interest). 
It is finally seen that the tower of $M_k$'s is the basic construction
tower for (the finite index, extremal $II_1$-subfactor) $M_0 \subseteq M_1$ and that the planar algebra of this
subfactor is naturally isomorphic to the original planar algebra $P$.
The proofs all rely on techniques of free probability and random matrices and indeed, one of the stated goals of the paper is to
demonstrate the connections between these and planar algebras.

The {\it raison d'\^{e}tre} of this paper is to demonstrate the power
of planar algebra techniques. 
We begin with a short summary of our notations and conventions regarding planar algebras in Section 2.
In Section 3, we describe  a tower $F_k(P)$ of filtered $*$-algebras, with compatible traces and `conditional expectations', associated to a subfactor planar algebra $P$.
The positivity of the traces being obvious, we show in Section 4  that the
associated GNS representations are bounded and thus yield
a tower of finite von Neumann algebra completions.
The heart of this paper is Section 5 which is devoted to showing that these completions are factors and to computations of some relative commutants. The penultimate Section 6 identifies the tower as a basic construction tower of a finite index extremal subfactor with  associated planar
algebra as the original  $P$.
The final section exhibits  interesting trace preserving $*$-isomorphisms between
the algebras $Gr_k(P)$ of \cite{GnnJnsShl} and our $F_k(P)$, thus
justifying - to some extent - the first sentence of this paper.

All our proofs rely solely on planar algebra techniques and in that sense
our paper is mostly self-contained.  
We will need neither the embedding theorem for a subfactor planar
algebra into the planar algebra of a bipartite graph nor any
free probability or random matrix considerations.
In the trade-off between analytic techniques and algebraic/combinatorial techniques that is characteristic of subfactor theory, it would be fair to say that \cite{GnnJnsShl} leans towards the analytic approach while this paper takes
the opposite tack.

After this paper had been written, we communicated it to Jones
requesting his comments. He wrote back saying that this ``may be similar to \cite{JnsShlWlk}" which appears on his homepage; and we discovered that this is indeed the case.

\section{Subfactor planar algebras}

The purpose of this section is to fix our notations and conventions
regarding planar algebras. We assume that the reader is familiar
with planar algebras as in \cite{Jns} or in \cite{KdySnd} so we will
be very brief.

Recall that the basic structure that underlies planar algebras is an action by the
`coloured operad of planar tangles' which concept we will now explain.
Consider the set $Col = \{0_+, 0_-, 1, 2, \cdots \}$, whose elements  are called colours.

We will not define a tangle but merely note the following features.
Each tangle has an external box, denoted $D_0$, 
and a (possibly empty) ordered collection of 
internal non-nested boxes denoted
$D_1$, $D_2$, $\cdots$.
Each box has an even number (again possibly 
0) of points marked on its boundary. A box with $2n$ points on its boundary is called an $n$-box or said to be of colour $n$.
There is also given a collection of disjoint curves each of which is either 
closed,
or joins a marked point on one of the boxes to another such.
For each box having at least one marked point on its boundary, one
of the regions ( = connected components of the complement of the boxes and
curves) that impinge on its boundary is distinguished and marked
with a $*$ placed near its boundary.
The whole picture is to be planar and each marked point on a box must be
the end-point of one of the curves.
Finally, there is given a chequerboard shading of the regions such 
that 

the $*$-region of any box is shaded white.
A $0$-box is said to be  $0_+$ box if the region touching its boundary is 
white and a  $0_-$ box
otherwise.
A $0$ without the $\pm$ qualification will always refer to $0_+$.
A tangle is said to be an $n$-tangle if its external box is of colour $n$.
Tangles are defined only upto
a planar isotopy preserving the $*$'s, the shading and the ordering of the 
internal boxes.

We illustrate several important tangles in Figure \ref{fig:imptangles}.
\begin{figure}[!htb]
\psfrag{n}{\tiny $n$}
\psfrag{i}{\tiny $i$}
\psfrag{$ER_{n+i}^n$ : Right expectations}{$ER_{n+i}^n$ : Right expectations}
\psfrag{2n}{\tiny $2n$}
\psfrag{2}{\tiny $2$}
\psfrag{D_1}{\small $D_1$}
\psfrag{D_2}{\small $D_2$}
\psfrag{$I_n^{n+1}$ : Inclusion}{$I_n^{n+1}$ : Inclusion}
\psfrag{$EL(i)_{n+i}^{n+i}$ : Left expectations}{$EL(i)_{n+i}^{n+i}$ : Left expectations}
\psfrag{$M_{n,n}^n$ : Multiplication}{$M_{n,n}^n$ : Multiplication}
\psfrag{$TR_n^0$ : Trace}{$TR_n^0$ : Trace}
\psfrag{$R^{n+1}_{n+1}$ : Rotation}{$R_{n+1}^{n+1}$ : Rotation}
\psfrag{$1^n$ : Multiplicative identity}{$1^n$ : Mult. identity}
\psfrag{$ER_{n+1}^{n+1}$ : Right expectation}{$ER_{n+1}^{n+1}$ : Right expectation}
\psfrag{$I_n^n$ : Identity}{$I_n^n$ : Identity}
\psfrag{$E^{n+2}$ : Jones projections}{$E^{n+2}$ : Jones proj.}
\includegraphics[height=8.5cm]{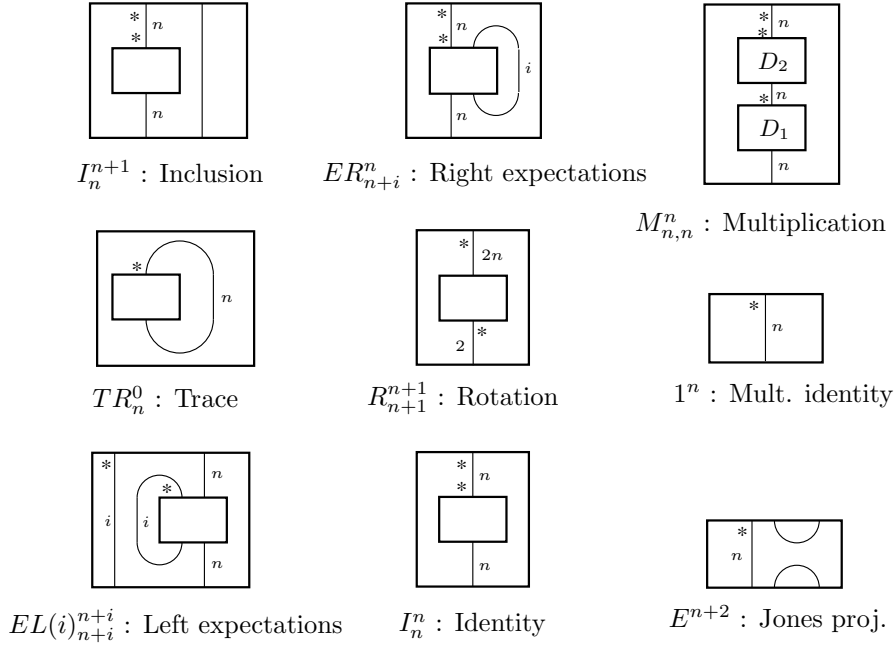}
\caption{Some important tangles (defined for $n \geq 0,0 \leq i \leq n$).}
\label{fig:imptangles}
\end{figure}
This figure (along with others in this paper)  uses the following notational device for convenience in
drawing tangles.
A strand in a tangle with a {\em non-negative integer}, say $t$, adjacent to it will indicate
a $t$-cable of that strand, i.e., a parallel cable of $t$ strands, in
place of the one actually drawn. Thus for instance, the tangle
equations of Figure \ref{cabledef} hold.
\begin{figure}[!htb]
\psfrag{3}{\tiny $3$}
\psfrag{6}{\tiny $6$}
\includegraphics[height=2cm]{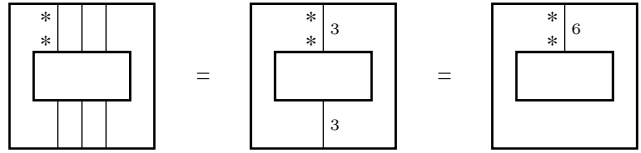}
\caption{Illustration of cabling notation for tangles}
\label{cabledef}
\end{figure}
In the sequel, we will have various integers adjacent to strands
in tangles and leave it to the reader to verify in each case that the
labelling integers are indeed non-negative.

A useful labelling convention for tangles (that we will not be very consistent in using though) is to decorate its tangle symbol,
such as $I,EL,M$ or $TR$, with subscripts and a superscript that give the colours
of its internal boxes and external box respectively. With this, we may dispense
with showing the shading, which is then unambiguously determined.
Another useful device is to label the marked points on any $n$-box 
with the numbers $1,2,\cdots,2n-1,2n$ in a clockwise fashion 
so that the interval from $2n$ to $1$ is in its $*$-region.

The basic operation that one can perform on tangles is substitution of one
into a box of another. If $T$ is a tangle that has some internal  boxes
$D_{i_1}, \cdots, D_{i_j}$ of colours $n_{i_1}, \cdots, n_{i_j}$ and if
$S_1, \cdots, S_j$ are arbitrary tangles of colours $n_{i_1}, \cdots, 
n_{i_j}$, then we may substitute $S_t$ into the box $D_{i_t}$ of $T$ for 
each $t$ - such that the `$*$'s match' - to get a new tangle that will be 
denoted $T \circ_{(D_{i_1}, \cdots,D_{i_j})}
(S_1, \cdots, S_j)$.
The collection of tangles along with the substitution operation is called
the coloured operad of planar tangles.

A planar algebra $P$ is an algebra over the coloured operad of planar tangles.
By this, is meant the following: $P$ is a collection $\{P_n\}_{n \in Col}$ of vector spaces
and linear maps $Z_T : P_{n_1} \otimes P_{n_2} \otimes \cdots \otimes 
P_{n_b} 
\rightarrow P_{n_0}$ for each $n_0$-tangle $T$ with internal boxes of colours
$n_1,n_2, \cdots,n_b$. The collection of maps is to be `compatible with 
substitution of tangles and renumbering of internal boxes' in an obvious 
manner.
For a planar algebra $P$, each $P_n$ acquires the structure of an
associative, unital algebra with multiplication defined using the
tangle $M_{n,n}^n$ and unit defined to be $1_n = Z_{1^n}(1)$.

Among planar algebras, the ones that we will be interested in are the
subfactor planar algebras.
These are complex, finite-dimensional and connected in the sense that each $P_n$ is a 
finite-dimensional complex vector space and $P_{0_\pm}$ are one dimensional.
They have a positive modulus $\delta$, meaning that closed loops in a 
tangle $T$ contribute
a multiplicative factor of $\delta$ in $Z_T$.
They are spherical in that for a $0$-tangle $T$, the function $Z_T$ is not 
just
planar isotopy invariant but also an isotopy invariant of the tangle regarded
as embedded on the surface of the two sphere.
Further, each $P_n$ is a $C^*$-algebra in such a way that for an $n_0$-tangle 
$T$ with internal boxes of colours
$n_1,n_2, \cdots,n_b$ and for $x_i \in P_{n_i}$, the equality
$Z_T(x_1 \otimes \cdots \otimes x_b)^* ~=~
Z_{T^*}(x_1^* \otimes \cdots \otimes x_b^*)$ holds,
where $T^*$ is the adjoint of the tangle $T$ - which, by definition, is 
obtained from $T$ by reflecting it.

Finally, the trace $\tau : P_n \rightarrow {\mathbb C} = P_{0}$ 
defined by:
\begin{eqnarray*}
\tau(x) ~=~  \delta^{-n} Z_{TR_n^{0}}(x)
\end{eqnarray*}
is postulated to be a faithful, positive (normalised) trace for each $n \geq 0$.

We then have the following fundamental thorem of Jones in \cite{Jns}.

\begin{theorem}\label{jones}
Let  
\[ (M_0 =) N \subset M (= M_1) \subset^{e_2} M_2 \subset \cdots \subset^{e_n} M_n 
\subset^{e_{n+1}} \cdots \]
be the tower of the basic construction associated to an extremal subfactor 
with index $[M:N] = \delta^2 < \infty$. Then there exists a unique subfactor planar 
algebra
$P = P^{N \subset M}$ of modulus ~$\delta$ satisfying the following conditions:

(0) $P^{N \subset M}_n = N^\prime \cap M_{n} ~\forall n \geq 0$ - where this
is regarded as an equality of *-algebras which is consistent with the 
inclusions on the two sides;

(1) $Z_{E^{n+1}}(1) = \delta ~e_{n+1} ~\forall ~n \geq 1$;

(2) $Z_{EL(1)^{n+1}_{n+1}}(x) = \delta ~E_{M^\prime \cap M_{n+1}}(x) ~\forall
~x \in N^\prime \cap M_{n+1}, ~\forall n \geq 0$;

(3) $Z_{ER_{n+1}^n}(x) = \delta ~E_{N^\prime \cap M_{n}}(x) ~\forall
~x \in N^\prime \cap M_{n+1}$; and this (suitably interpreted for $n = 0_\pm$)  is required to hold for all
$n \in Col$.

Conversely, any subfactor planar algebra $P$ with modulus ~$\delta$ arises 
from an extremal subfactor of index $\delta^2$ in this fashion. \qed
\end{theorem}

Recall that a finite index $II_1$-subfactor $N \subseteq M$ is said to be
extremal if the restriction of the traces on $N^\prime$ (computed in ${\mathcal L}(L^2(M))$) and $M$
to $N^\prime \cap M$ agree. 
The notations $E_{M^\prime \cap M_{n+1}}$ and $E_{N^\prime \cap M_{n}}$ stand for the trace preserving
conditional expectations of $N^\prime \cap M_{n+1}$ onto $M^\prime \cap M_{n+1}$ and $N^\prime \cap M_{n}$ respectively.

The converse part of Jones' theorem is, in essence, the result of Popa  alluded to in the introduction,
as remarked in \cite{Jns}.
It is this converse, as proved in \cite{GnnJnsShl}, for which we supply 
a different proof in the rest of this paper.

We note that the multiplication tangle here agrees
with the one in \cite{GnnJnsShl} but is adjoint to the ones in \cite{Jns} and in \cite{KdySnd}
while the rotation tangle is as in \cite{KdySnd} but adjoint to the one
in \cite{Jns}.

\begin{remark}\label{plalg}
For a subfactor planar algebra $P$,
(i) the right expectation tangles $ER_{n+i}^n$
 define  surjective positive maps $P_{n+i} \rightarrow P_n$  of norm $\delta^i$,
(ii) the left expectation tangles $EL(i)_{n+1}^{n+1}$ define  positive maps $P_{n+1} \rightarrow P_{n+1}$ whose images, denoted $P_{i,n+1}$,
are $C^*$-subalgebras of $P_{n+1}$ and
(iii) the rotation tangles $R^{n+1}_{n+1}$ define unitary maps $P_{n+1} \rightarrow P_{n+1}$.
The first two statements follow from the fact that
$ER^n_{n+i}$ and $EL(i)_{n+1}^{n+1}$  give (appropriately scaled) conditional
expectation maps that preserve a faithful, positive trace while the third is a
consequence of the compatibility of the tangle $*$ and the $*$ of the $C^*$-algebra
$P_{n+1}$.
\end{remark}

\section{The tower of filtered $*$-algebras with trace}

For the rest of this paper, the following notation will hold. Let
$P$ be a subfactor planar algebra of modulus $\delta > 1$.
We therefore have finite-dimensional $C^*$-algebras $P_n$ for $n \in Col$
with appropriate inclusions. 
For
$x \in P_n$, set $||x||_{P_n} = \tau(x^*x)^\frac{1}{2}$; this defines a norm on
$P_n$.

For $k \geq 0$, let $F_k(P)$ be the 
vector space direct sum $\oplus_{n=k}^{\infty} P_n$ (where $0
= 0_+$, here and in the sequel).
Our goal, in this section, is to equip each $F_k(P)$ with a
filtered, associative, unital $*$-algebra structure with  normalised trace $t_k$
and to describe trace preserving filtered $*$-algebra inclusions
$F_0(P) \subseteq F_1(P) \subseteq F_2(P) \subseteq \cdots$,
as well as conditional expectation-like maps 
$F_0(P) \stackrel{E_0}{\leftarrow} F_1(P) \stackrel{E_1}{\leftarrow}
F_2(P) \stackrel{E_2}{\leftarrow} \cdots$.

We begin by defining the multiplication.
For $a \in F_k(P)$, we will denote by $a_n$, its $P_n$ component
for $n \geq k$. Thus $a = \sum_{n=k}^\infty a_n = (a_k,a_{k+1},\cdots) \in F_k(P)$, where
only finitely many $a_n$ are non-zero.
Now suppose that $a =  a_m \in P_m$ and $b = b_n \in P_n$ where $m,n \geq k$. Their
product in $F_k(P)$, denoted\footnote{Rather than being notationally correct and write $\#_k$, we drop the subscript in the interests of aesthetics.} $a\#b$, is defined to be $\sum_{t=|n-m|+k}^{n+m-k} (a\#b)_t$ where $(a\#b)_t$ is given by
the tangle in Figure \ref{sharpdef}.
\begin{figure}[!htb]
\psfrag{a}{$a$}
\psfrag{b}{$b$}
\psfrag{m+t-n-k}{\tiny $m\!+\!t-$}
\psfrag{m+n-k-t}{\tiny $m\!+\!n\!-\!k\!-\!t$}
\psfrag{n+t-m-k}{\tiny $n\!+\!t-$}
\psfrag{-m-k}{\tiny $m\!-\!k$}
\psfrag{-n-k}{\tiny $n\!-\!k$}
\psfrag{k}{\tiny $k$}
\includegraphics[height=2.5cm]{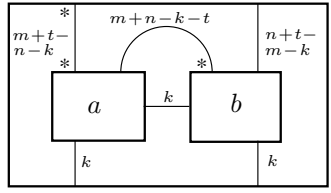}
\caption{Definition of the $P_t$ component of $a\#b$.}
\label{sharpdef}
\end{figure}
Define $\#$ by extending this map bilinearly to the whole of $F_k(P) \times F_k(P)$.

As in \cite{GnnJnsShl}, we will reserve `$*$' to denote the usual involution
on $P_n$'s and use $\dagger$ (rather than $\dagger_k$)  to denote the involution on $F_k(P)$.
For $a = a_m  \in P_m \subseteq F_k(P)$ define $a^\dagger \in P_m$ by the
tangle in Figure \ref{daggerdef}
\begin{figure}[!htb]
\psfrag{2(m-k)}{\tiny $2(m\!-\!k)$}
\psfrag{2k}{\tiny $2k$}
\psfrag{a^*}{$a^*$}
\includegraphics[height=2cm]{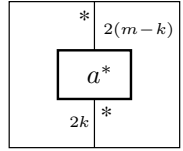}
\caption{Definition of the $*$-structure on $F_k(P)$.}
\label{daggerdef}
\end{figure}
and extend additively to the whole of $F_k(P)$.
Note that $a^\dagger = Z_{(R^m_m)^k}(a^*)$ where $R_m^m$ is the 
$m$-rotation tangle and $(R^m_m)^k = R_m^m \circ R_m^m \circ \cdots \circ R_m^m$ ($k$ factors).

Next, define the linear functional $t_k$ on $F_k(P)$ to be the normalised  trace of
its $P_k$ component, i.e., for $a=(a_k,a_{k+1},\cdots) \in F_k(P)$
define $t_k(a) = \tau(a_k)$.

Then, define the `inclusion map' of $F_{k-1}(P)$
into $F_{k}(P)$ (for $k \geq 1$)
as the map whose restriction takes $P_{n-1} \subseteq F_{k-1}(P)$ to $P_{n} \subseteq
F_{k}(P)$ by the 
tangle illustrated in Figure~\ref{incldef}.
\begin{figure}[!htb]
\psfrag{k}{\tiny $k\!-\!1$}
\psfrag{2n-k}{\tiny $2n-$}
\psfrag{-k-1}{\tiny $k\!-\!1$}
\includegraphics[height=2cm]{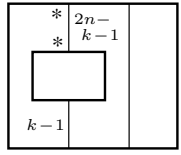}
\caption{
The inclusion from $P_{n-1} \subseteq F_{k-1}(P)$
to $P_{n} \subseteq F_{k}(P)$.}
\label{incldef}
\end{figure}

Finally, define the conditional expectation-like map $E_{k-1} : F_k(P)
\rightarrow F_{k-1}(P)$ (for $k \geq 1$) as $\delta^{-1}$ times the map whose restriction takes
$P_n  \subseteq F_k(P)$ to $P_{n-1} \subseteq
F_{k-1}(P)$ by the tangle illustrated in Figure \ref{condexpdef}.
\begin{figure}[!htb]
\psfrag{k-1}{\tiny $k\!-\!1$}
\psfrag{2n-k-1}{\tiny $2n\!-\!k$}
\psfrag{-1}{\tiny $-\!1$}
\includegraphics[height=2cm]{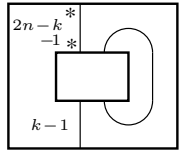}
\caption{Definition of $\delta E_{k-1}$ from $P_n \subseteq F_k(P)$
to $P_{n-1} \subseteq F_{k-1}(P)$.}
\label{condexpdef}
\end{figure}

An important observation that we use later is that `restricted' to $P_k \subseteq F_k(P)$,
all these maps are the `usual' ones of the planar algebra $P$.
More precisely we have (a) $\#|_{P_k \otimes P_k} = Z_{M^k_{k,k}}$,
(b) $\dagger|_{P_k} = *$, (c) $t_k|_{P_k} = \tau$, (d) The inclusion
of $F_{k-1}(P)$ into $F_k(P)$ restricts to $Z_{I_{k-1}^k}: P_{k-1}
\rightarrow P_k$, and (e) $\delta E_{k-1}|_{P_k} = Z_{ER^{k-1}_k}$.

We can now state the main result of this section.
\begin{proposition}\label{filtalg}
Let $k \geq 0$. Then, the following statements hold.
\begin{enumerate}
\item The vector space $F_k(P)$ acquires a natural
 associative, unital,  filtered algebra structure for the $\#$ multiplication.
\item The operation $\dagger$ defines a conjugate linear, involutive, anti-isomorphism of $F_k(P)$ to itself thus making it a $*$-algebra.
\item The map $t_k$ on $F_k(P)$ satisfies the
equation $t_k(b^\dagger\#a) = \delta^{-k} \sum_{n=k}^\infty \delta^n\tau(b_n^*a_n)$ and consequently defines a normalised trace on $F_k(P)$ that makes
$\langle a | b \rangle = t_k(b^\dagger\#a)$
an inner-product on $F_k(P)$.
\item The inclusion map of $F_k(P)$ into $F_{k+1}(P)$ is a normalised trace-preserving unital $*$-algebra monomorphism.
\item The map $E_{k} : F_{k+1}(P) \rightarrow F_k(P)$ is a
$*$- and trace-preserving $F_k(P)-F_k(P)$ bimodule
retraction for the inclusion map of $F_k(P)$ into $F_{k+1}(P)$.
\end{enumerate}
\end{proposition}

\begin{proof}
(1) 
Take
$a \in P_m, b \in P_n, c \in P_p$ with $m,n,p \geq k$.
Then, by definition of $\#$,
\begin{eqnarray*}
(a\#b)\#c &=& \sum_{t=|m-n|+k}^{m+n-k} (a\#b)_t\#c
= \sum_{t=|m-n|+k}^{m+n-k} \sum_{s=|t-p|+k}^{t+p-k} ((a\#b)_t\#c)_s, {\text {~~while}}\\
a\#(b\#c) &=& \sum_{v=|n-p|+k}^{n+p-k} a\#(b\#c)_v =
\sum_{v=|n-p|+k}^{n+p-k} \sum_{u=|m-v|+k}^{m+v-k} (a\#(b\#c)_v)_u.
\end{eqnarray*}

Let $I(m,n,p) = \{(t,s) : |m-n|+k \leq t \leq m+n-k, |t-p|+k \leq s \leq t+p-k\}$ and $J(m,n,p) = \{(v,u) : |n-p|+k \leq v \leq n+p-k, |m-v|+k \leq u \leq m+v-k\}$, so that the former indexes terms of $(a\#b)\#c$ while the
latter indexes terms of $a\#(b\#c)$. 

A routine verification shows that (a) $I(m,n,p) = J(p,n,m)$ and (b)
the map $T(m,n,p): I(m,n,p) \rightarrow J(m,n,p)$ defined by
$T(m,n,p)((t,s)) = (max\{m+p,n+s\}-t,s)$ is a well-defined bijection
with inverse $T(p,n,m)$ such that (c) if $T(m,n,p)((t,s)) = (v,u)$, then both $((a\#b)_t\#c)_s$ and $(a\#(b\#c)_v)_u$ are equal to the figure on the right or on the
left in Figure \ref{assoc}
\psfrag{a}{$a$}
\psfrag{b}{$b$}
\psfrag{c}{$c$}
\psfrag{k}{\tiny $k$}
\psfrag{k-1}{\tiny $k\!-\!1$}
\psfrag{m+n-t-k}{\tiny $m\!+\!n\!-\!t\!-\!k$}
\psfrag{n+m-t-k}{\tiny $m\!+\!n\!-\!t\!-\!k$}
\psfrag{n-m+t-k}{\tiny $n\!-\!m\!+\!t\!-\!k$}
\psfrag{m+p-n-s}{\tiny $m\!+\!p\!-\!n\!-\!s$}
\psfrag{n+s-m-p}{\tiny $n\!+\!s\!-\!m\!-\!p$}
\psfrag{t+p-s-k}{\tiny $t\!+\!p\!-\!s\!-\!k$}
\psfrag{m+n-t-k-1}{\tiny $m\!+\!n\!-\!t\!-\!k\!-\!1$}
\psfrag{n-m+t-k-1}{\tiny $n\!-\!m\!+\!t\!-\!k\!-\!1$}
\psfrag{p-t+s}{\tiny $p\!-\!t\!+\!s$}
\psfrag{p-t+}{\tiny $p\!-\!t+$}
\psfrag{p+u-}{\tiny $p\!+\!u-$}
\psfrag{t+u-}{\tiny $t\!+\!s-$}
\psfrag{m-n+}{\tiny $m\!-\!n+$}
\psfrag{t-k}{\tiny $t\!-\!k$}
\psfrag{s-k}{\tiny $s\!-\!k$}
\psfrag{p-k}{\tiny $p\!-\!k$}
\begin{figure}[!htb]
\includegraphics[height=2.7cm]{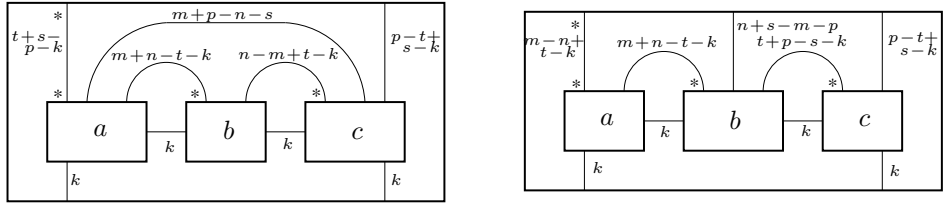}
\caption{$((a\#b)_t\#c)_s  = (a\#(b\#c)_v)_u$.}
\label{assoc}
\end{figure}
according as $m+p \geq n+s$ or $m+p \leq n+s$; this finishes the proof of associativity.

Observe that the usual unit $1_k$ of $P_k$ is also the unit for
the $\#$-multiplication in $F_k(P)$ and that there is an obvious
filtration of $F_k(P)$ by the subspaces which are the direct sums of
the first $n$ components of $F_k(P)$ for $n \geq 0$.\\
(2) This is an entirely standard pictorial planar algebra argument
which we'll omit.\\
(3) By definition, $t_k(b^\dagger\#a)$ is the normalised trace of
$(b^\dagger\#a)_k$. By definition of $\#$, the $k$-component
of $b^\dagger\#a$ has contributions only from $b_n^\dagger\#a_n$
for $n \geq k$, and the normalised trace of this contribution is given by
$\delta^{-k+n}\tau(b_n^*a_n)$, as needed. 
The fact that $t_k$ is a normalised trace on $F_k(P)$ follows from this
and the unitarity of $Z_{R_n^n}$ on $P_n$.
It also follows that $\langle a | b \rangle = t_k(b^\dagger\#a)$ is an inner-product on $F_k(P)$.\\
(4), (5) We also omit these proofs which  are routine applications of
pictorial techniques using the definitions.
\end{proof}

We define $H_k$ to be the Hilbert space completion of $F_k(P)$.
Note that for the inner-product on $F_k(P)$, the subspaces $P_n$ 
of $H_k$
are mutually orthogonal and so $H_k$ is their orthogonal direct
sum $\oplus_{n=k}^\infty P_n$. In particular, elements of $H_k$
are sequences $\xi = (x_k,x_{k+1},\cdots)$ where $||\xi||_{H_k}^2 = \delta^{-k} \sum_{n=k}^\infty \delta^{n} \tau(x_n^*x_n) < \infty$.

\section{Boundedness of the GNS representations}

Our goal in this section is to show that for each $k \geq 0$, the 
left and right regular representations of $F_k(P)$ are both
bounded for the norm on $F_k(P)$ and therefore extend uniquely to
$*$-homomorphisms $\lambda_k, \rho_k : F_k(P) \rightarrow {\mathcal L}(H_k)$. We then show that the von Neumann algebras
generated by $\lambda_k(F_k(P))$ and $\rho_k(F_k(P))$ are finite
and commutants of each other.
Finally we show that for $k \geq 0$, the finite von Neumann algebras
$\lambda_k(F_k(P))''$ naturally form a tower.

The key estimate we need for proving boundedness is contained
in Proposition \ref{prop:estimate}, the proof of which appeals to the
following lemma.

\begin{lemma}\label{lem:estimlem}
Fix $p \geq k$.
Suppose that $0 \leq q \leq 2p$, $0 \leq i \leq 2p-q$ and $a \in P_p \subseteq H_k$.
There is a unique positive $c \in P_{2p-q}$ such that the equation
\begin{figure}[!htb]
\psfrag{a}{$a$}
\psfrag{c}{$c$}
\psfrag{a^*}{$a^*$}
\psfrag{c^*}{$c^*$}
\psfrag{=}{$=$}
\psfrag{i}{\tiny $i$}
\psfrag{q}{\tiny $q$}
\psfrag{2p-q-i}{\tiny $2p-q-i$}
\psfrag{2p-q}{\tiny $2p-q$}
\psfrag{m+t-n-k}{\tiny $m\!+\!t\!-\!n\!-\!k$}
\psfrag{n+t-m-k}{\tiny $n\!+\!t\!-\!m\!-\!k$}
\psfrag{m+n+k-t}{\tiny $u$}
\psfrag{m+n+k-t = u}{\tiny $u=m\!+\!n\!+\!k\!-\!t$}
\includegraphics[height=3cm]{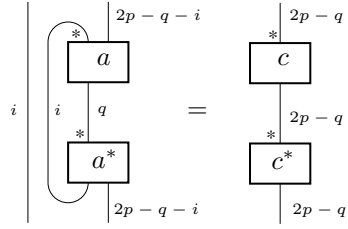}
\caption{On positivity}
\label{estimlem}
\end{figure}
of Figure \ref{estimlem} holds. Further, $||c||_{H_k}^2 =\delta^i  ||a||_{H_k}^2$.
\end{lemma}

\begin{proof} To show the existence and uniqueness of $c$,
it clearly suffices to see that the picture on the left in Figure \ref{estimlem} defines a positive element of the $C^*$-algebra $P_{2p-q}$, for then, $c$ must be its positive square root.
But this element may be written as $Z_{EL(i)_{2p-q}^{2p-q}}(b)$
where $b$ is illustrated in Figure \ref{bfig}, and so by Remark \ref{plalg}
it suffices to see that $b$ itself is positive. The pictures on the
\begin{figure}[!htb]
\psfrag{or}{or $(ii)$}
\psfrag{a}{$a$}
\psfrag{b}{$b=$}
\psfrag{a^*}{$a^*$}
\psfrag{c^*}{$c^*$}
\psfrag{d}{$= (i)\ \delta^{q-p}$}
\psfrag{i}{\tiny $i$}
\psfrag{k}{\tiny $q$}
\psfrag{m-k}{\tiny $p-q$}
\psfrag{q-p}{\tiny $q-p$}
\psfrag{2m-k}{\tiny $2p-q$}
\psfrag{2p-q}{\tiny $2p-q$}
\psfrag{m}{\tiny $p$}
\psfrag{p}{\tiny $p$}
\psfrag{n+t-m-k}{\tiny $n\!+\!t\!-\!m\!-\!k$}
\psfrag{m+n+k-t}{\tiny $u$}
\psfrag{m+n+k-t = u}{\tiny $u=m\!+\!n\!+\!k\!-\!t$}
\includegraphics[height=3.5cm]{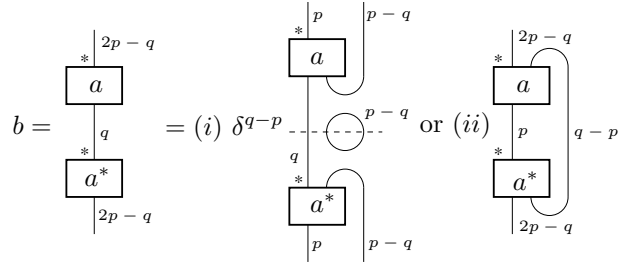}
\caption{Positivity of $b$}
\label{bfig}
\end{figure}
right in Figure \ref{bfig} exhibit $b$ as (i) a positive multiple of a product of an element of
$P_{2p-q}$ (the one below the dotted line) and its adjoint (the one above) if $q \leq p$ or as (ii) $Z_{ER_p^{2p-q}}(a^*a)$ (the image of  a positive element under a positive map) if $q \geq p$, proving that $b$ is positive in either case. The norm statement follows by applying the trace tangle $TR^0_{2p-q}$
 on both sides in Figure~\ref{estimlem}
and recalling that  $\delta^{k-u} ||x||^2_{H_k} = \tau(x^*x)$ for $x \in P_u \subseteq H_k$.
\end{proof}

\begin{proposition}\label{prop:estimate}
Suppose that $a = a_m \in P_m \subseteq F_k(P)$. There exists
a constant $K$ (depending only on $a$) so that, for all $b = b_n \in P_n
\subseteq F_k(P)$ and all $t$ such that $|m-n|+k \leq t \leq m+n-k$, we have $|| (a\#b)_t ||_{H_k} \leq K ||b||_{H_k}$.
\end{proposition}

\begin{proof} 
By definition, $||(a\#b)_t||^2_{H_k}
= \delta^{-k} \delta^t \tau((a\#b)_t(a\#b)_t^*)$, and so $\delta^k ||(a\#b)_t||^2_{H_k}$ is given by -
using Figure \ref{sharpdef} -  the value of the tangle on the left in Figure \ref{estimate3}. Setting $\tilde{b} = Z_{(R^n_n)^{-k}}(b)$, this is also seen to be equal to the value of the
tangle in the middle in Figure \ref{estimate3}.
\begin{figure}[!htb]
\psfrag{a}{$a$}
\psfrag{=}{$=$}
\psfrag{b}{$b$}
\psfrag{a*}{$a^*$}
\psfrag{b*}{$b^*$}
\psfrag{Zb}{$\tilde{b}$}
\psfrag{Zb*}{$\tilde{b}^*$}
\psfrag{c}{$c_u$}
\psfrag{d}{$d_u^*$}
\psfrag{c*}{$c^*_u$}
\psfrag{bt}{$\tilde{b}$}
\psfrag{bt*}{$\tilde{b}^*$}
\psfrag{d*}{$d_u$}
\psfrag{k}{\tiny $k$}
\psfrag{2m-k}{\tiny $2m-k$}
\psfrag{n-k}{\tiny $n-k$}
\psfrag{m+t-n-k}{\tiny $m\!+\!t-$}
\psfrag{n+t-m-k}{\tiny $n\!+\!t\!-\!m\!-\!k$}
\psfrag{u}{\tiny $u$}
\psfrag{m+n-k-t}{\tiny $m\!+\!n\!-\!k\!-\!t$}
\includegraphics[height=5cm]{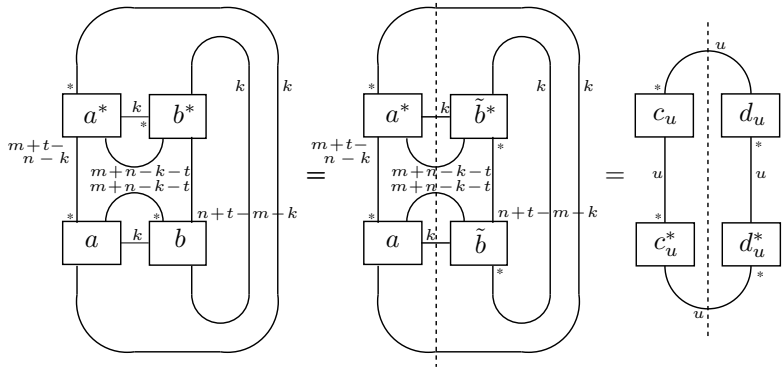}
\caption{$\delta^k ||(a\#b)_t||^2_{H_k}$}
\label{estimate3}
\end{figure}

Now let $u = m+n+k-t$ and note that $2k \leq u \leq min\{2m,2n\}$.
We now apply Lemma \ref{lem:estimlem} to the tangle on the left of the dotted line
(with $i=0$, $p=m$ and $q=m-n+t-k$) and to the (inverted) tangle on the right of the
dotted line (with $i=k$, $p=n$ and $q=n+t-m-k$) to conclude that there exist $c_u,d_u \in P_{u}$
such that $||c_u||_{H_k}^2 = ||a||_{H_k}^2$, $||d_u||_{H_k}^2 = \delta^k
||\tilde{b}||_{H_k}^2 = \delta^k
||b||_{H_k}^2$ and the second equality in Figure \ref{estimate3}
holds.
Therefore, $||(a\#b)_t||^2_{H_k} = \delta^{-k} \delta^{u} \tau(c_u^*
c_ud_ud_u^*) = \delta^{u-k} ||c_ud_u||^2_{P_{u}}$ $\leq \delta^{u-k} ||c_u||_{op}^2 ||d_u||^2_{P_{u}} =  ||c_u||_{op}^2 ||d_u||^2_{H_k} = \delta^k ||c_u||_{op}^2 ||b||^2_{H_k}$. Here, we write  $||x||^2_{P_u} = \tau(x^*x)$ and $||x||_{op}$ for the operator norm of $x$ in the $C^*$-algebra $P_u$.

Finally, take $K$ to be the maximum of $\delta^{\frac{k}{2}}||c_u||_{op}$
as $u$ varies between $2k$ and $2m$, which clearly depends only
on $a$.
\end{proof}

We are ready to prove boundedness of the left regular representation.
During the course of the proof we will need to use the fact that 
$||\sum_{i=1}^k a_i||^2 \leq k(\sum_{i=1}^k ||a_i||^2)$
for
vectors $a_1,\cdots,a_k$ in  an inner-product space, 
which follows on applying Cauchy-Schwarz to the vectors $(||a_1||,\cdots,||a_k||)$  and $(1,1,\cdots,1)$ in ${\mathbb C}^k$.

\begin{proposition}\label{prop:boundedness}
Suppose that $a \in P_m \subseteq F_k(P)$ and $\lambda_k(a) :
F_k(P) \rightarrow F_k(P)$ is defined by $\lambda_k(a)(b) = a\#b$.
Then there is a constant $C$ (depending only on $a$) so that $||\lambda_k(a)(b)||_{H_k} \leq
C||b||_{H_k}$ for all $b \in F_k(P)$.
\end{proposition}

\begin{proof}
Suppose that $b = \sum_{n=k}^\infty b_n$ (the sum being finite, of
course). Then $a\#b = \sum_{t=k}^\infty \sum_{n=k}^\infty (a\#b_n)_t
= \sum_{t=k}^\infty \sum_{n = |t-m| + k}^{t+m-k}  (a\#b_n)_t$, where
the last equality is by definition of where non-zero terms of $a\#b_n$ may  lie. Thus $(a\#b)_t = \sum_{n = |t-m| + k}^{t+m-k}  (a\#b_n)_t$ is a sum
of at most $1 + min\{2(t-k),2(m-k)\}$ terms.
By the remark preceding the proposition, it follows that
\begin{eqnarray*}
||(a\#b)_t||^2 _{H_k} &\leq& (1 + 2(m-k)) \sum_{n = |t-m| + k}^{t+m-k}  ||(a\#b_n)_t||^2_{H_k}\\
&\leq& K^2(1 + 2(m-k)) \sum_{n = |t-m| + k}^{t+m-k}  ||b_n||^2_{H_k}
\end{eqnarray*}
where $K$ is as in Proposition \ref{prop:estimate}.
Therefore,
\begin{eqnarray*}
||a\#b||^2_{H_k} &=&  \sum_{t=k}^\infty
||(a\#b)_t||^2_{H_k}\\
&\leq&  K^2(1 + 2(m-k)) \sum_{t=k}^\infty  \sum_{n = |t-m| + k}^{t+m-k} ||b_n||^2_{H_k}\\
&=&  K^2(1 + 2(m-k)) \sum_{n=k}^\infty \sum_{t = |n-m| +k}^{n+m-k} ||b_n||^2_{H_k}\\
&=&  K^2(1 + 2(m-k)) \sum_{n=k}^\infty (1+min\{2(n-k),2(m-k)\}) ||b_n||^2_{H_k}\\
&\leq& K^2(1 + 2(m-k))^2 ||b||^2_{H_k}.
\end{eqnarray*} So we may choose $C$ to be $K(1+2(m-k))$.
\end{proof}

It follows from Proposition \ref{prop:boundedness} that for
any $a \in F_k(P)$, the map $\lambda_k(a)$ is bounded and thus
extends uniquely to an element of ${\mathcal L}(H_k)$. We
then get a $*$-representation $\lambda_k : F_k(P) \rightarrow
{\mathcal L}(H_k)$.

In order to avoid repeating hypotheses we make the following
definition. By a finite pre-von Neumann algebra, we will mean
a complex $*$-algebra $A$ that is equipped with a normalised trace $t$ such that (i) the sesquilinear form defined by $\langle a^\prime|a \rangle = t(a^*a^\prime)$
defines an inner-product on $A$ and such that (ii) for each $a \in A$,  the left-multiplication map $\lambda_A(a) : A \rightarrow A$ is bounded for
the trace induced norm of $A$.
Examples that we will be interested in are the $F_k(P)$ with
their natural traces $t_k$ for $k \geq 0$.
We then have the following simple lemma which motivates the terminology and whose proof we sketch in some detail.

\begin{lemma}\label{single}
Let $A$ be a finite pre-von Neumann algebra with trace $t_A$,
and $H_A$ be the Hilbert space completion
of $A$ for the associated norm, so that the left regular representation $\lambda_A : A \rightarrow {\mathcal L}(H_A)$ 
is well-defined, i.e., for each $a \in A$, $\lambda_A(a) : A \rightarrow A$ extends to a bounded operator on $H_A$. Then, 
\begin{enumerate} 
\item [(0)] The right regular representation
$\rho_A : A \  \rightarrow {\mathcal L}(H_A)$ is also well-defined.
\end{enumerate}
Let $M^\lambda_A = \lambda_A(A)''$ and $M^\rho_A = \rho_A(A)''$. Then
the following statements hold:
\begin{enumerate}
\item The `vacuum vector' $\Omega_A \in H_A$ (corresponding to $1 \in A \subseteq H_A$) is cyclic and separating for both $M^\lambda_A$ and $M^\rho_A$.
\item $M^\lambda_A$ and $M^\rho_A$ are commutants of each other.
\item The trace $t_A$ extends to faithful, normal, tracial states $t^\lambda_A$ and $t^\rho_A$ on $M^\lambda_A$ and $M^\rho_A$ respectively. 
\item 
$H_A$ can be identified with
the standard modules $L^2(M^\lambda_A,t^\lambda_A)$ and $L^2(M^\rho_A,t^\rho_A)$.
\end{enumerate}
\end{lemma}

\begin{proof}
Start with the modular conjugation operator $J_A$ which is
the unique bounded extension, to $H_A$, of the involutive, conjugate-linear, isometry defined on the dense subspace $A\Omega_A \subseteq H_A$ by $a\Omega_A \mapsto a^*\Omega_A$, and which satisfies
$J_A = J_A^* = J_A^{-1}$ (where $J^*$ is defined by $\langle J^*\xi | \eta \rangle
= \langle J\eta | \xi \rangle$ for conjugate linear $J$).\\
(0) Note that $J_A\lambda_A(a^*)J_A(\tilde{a}\Omega_A) = J_A\lambda_A(a^*)(\tilde{a}^*\Omega_A) = J_A(a^*{\tilde{a}}^*\Omega_A) = \tilde{a} a\Omega_A = \rho_a(\tilde{a}\Omega_A)$ and so for each $a \in A$, $\rho_A(a) : A \rightarrow A$ extends to the bounded operator $J_A\lambda_{A}(a^*)J_A$ on $H$.
The map $\lambda_A : A \rightarrow {\mathcal L}(H)$ is a $*$-homomorphism while
the map $\rho_A : A \rightarrow {\mathcal L}(H)$ is a $*$-anti-homomorphism.\\
(1) Since $\lambda_A(A) \subseteq {\mathcal L}(H_A)$ and $\rho_A(A) \subseteq {\mathcal L}(H_A)$ clearly commute, so do $M^\lambda_A$ and $M^\rho_A$. So each is contained
in the commutant of the other and it follows easily from definitions that $\Omega_A$ is cyclic and hence separating for both $M^\lambda_A$ and $M^\rho_A$.\\
(2) Observe now that the subspace $K = \{x \in {\mathcal L}(H_A) : J_Ax \Omega_A = x^* \Omega_A\}$ is weakly closed and 
contains $M^\lambda_A$, $M^\rho_A$ and their commutants. (Reason: That $K$ is weakly closed and contains $\lambda_A(A)$ and $\rho_A(A)$
is obvious, and so $K$ also contains their weak closures $M^\lambda_A$ and $M^\rho_A$. But now, for $x^\prime \in (M^\lambda_A)^\prime \cup
(M^\rho_A)^\prime$
and $a\Omega_A
\in A\Omega_A$, we have 
\begin{eqnarray*}
\langle J_A x^{\prime} \Omega_A | a \Omega_A \rangle &=& \langle J_A a \Omega_A | x^{\prime} \Omega_A \rangle\\
 &=&
\langle a^* \Omega_A | x^\prime \Omega_A \rangle\\
&=&  \langle \Omega_A |
x^\prime a\Omega_A \rangle\\
 &=& \langle x^{\prime *} \Omega_A | a \Omega_A \rangle,
\end{eqnarray*}
where the third equality follows from interpreting $a^*\Omega_A$
as either $\lambda_A(a)^*\Omega_A$ or $\rho_A(a)^*\Omega_A$
according as $x^\prime \in (M^\lambda_A)^\prime$ or $x^\prime \in (M^\rho_A)^\prime$.
Density of $A\Omega_A$ in $H$ now implies that $J_Ax^\prime \Omega_A =
x^{\prime *} \Omega_A$).

Since $K \supseteq M^\lambda_A$, it is easy to see that $J_A M^\lambda_A J_A \subseteq (M^\lambda_A)^\prime$, and similarly since $K \supseteq
(M^\lambda_A)^\prime$, we have $J_A (M^\lambda_A)^\prime J_A \subseteq 
(M^\lambda_A)^{\prime\prime} =
M^\lambda_A$. Comparing, we find indeed that $J_A M^\lambda_A J_A = (M^\lambda_A)^\prime$.
On the other hand, taking double commutants of
$J_A \lambda_A(A) J_A = \rho_A(A)$ gives $J_A M^\lambda_A J_A = M^\rho_A$ and
so $(M^\lambda_A)^\prime = M^\rho_A$.\\
(3) Define a linear functional $t$ on ${\mathcal L}(H_A)$ by $t(x) = \langle x\Omega_A | \Omega_A \rangle$ and observe that this extends the
trace $t_A$ on $A$ where $A$ is regarded as contained in ${\mathcal L}(H_A)$ via either $\lambda_A$ or $\rho_A$. Define $t^\lambda_A = t|_{M^\lambda_A}$ and $t^\rho_A = t|_{M^\rho_A}$. A little thought shows that
these are traces on $M^\lambda_A$ and $M^\rho_A$. Positivity is clear
and faithfulness is a consequence of the fact that $\Omega_A$ is 
separating for $M^\lambda_A$ and $M^\rho_A$. Normality (= continuity for the $\sigma$-weak topology) holds since
$t$ is a vector state on ${\mathcal L}(H_A)$.\\
(4) This is an easy consequence of (3).
\end{proof}

We conclude that if $M^\lambda_k = \lambda_k(F_k(P))'' \subseteq {\mathcal L}(H_k)$ and
$M^\rho_k = \rho_k(F_k(P))''  \subseteq {\mathcal L}(H_k)$ and $\Omega_k \in H_k$ is the
vacuum vector $1 = 1_k \in P_k \subseteq F_k(P) \subseteq H_k$,
then $M^\lambda_k$ and $M^\rho_k$ are finite von Neumann
algebras
(equipped with faithful, normal, tracial states $t^\lambda_k$ and
$t^\rho_k$)  that are commutants of each other and
$H_k$  can be identified as the standard module for both, with
$\Omega_k$ being a cyclic and separating trace-vector 
for both.  

In order to get 
a tower of von Neumann algebras we introduce a little more terminology.
By a compatible pair of finite pre-von Neumann algebras, we will
mean a pair $(A,t_A)$ and $(B,t_B)$ of finite pre-von Neumann algebras
such that $A \subseteq B$ and $t_B|_A = t_A$.
In particular, for any $k \geq 0$, $F_k(P) \subseteq F_{k+1}(P)$
equipped with their natural traces $t_k$ and $t_{k+1}$
give examples.
Given such a pair of compatible pre-von Neumann algebras,
identify $H_A$ with a subspace of
$H_B$ so that $\Omega_A = \Omega_B = \Omega$, say.

We will need the following lemma which is an easy consequence of
Theorem II.2.6 and Proposition III.3.12 of \cite{Tks}.

\begin{lemma}\label{normality}
Suppose that $M \subseteq {\mathcal L}(H)$ and $N \subseteq {\mathcal L}(K)$ are von Neumann algebras and 
$\theta: M \rightarrow N$ is a $*$-homomorphism
such that if $x_i \rightarrow x$ is a norm bounded strongly convergent net in $M$, then $\theta(x_i) \rightarrow \theta(x)$ strongly in $N$.
Then $\theta$ 
is a normal map and its image is
 a von Neumann subalgebra of $N$.\qed
\end{lemma}

\begin{proposition}\label{comppair}
Let $(A,t_A)$ and $(B,t_B)$ be a compatible pair of finite pre-von
Neumann algebras and 
$\Omega$ be
as above.
Let $\lambda_A : A \rightarrow {\mathcal L}(H_A)$ and $\lambda_B : B \rightarrow {\mathcal L}(H_B)$ be the left regular representations
of $A$ and $B$ respectively and let $M^\lambda_A = \lambda_A(A)''$
and $M^\lambda_B = \lambda_B(B)''$.
Then,
\begin{enumerate}
\item The inclusion $A \subseteq B$ extends uniquely to a normal
inclusion, say $\iota$, of $M^\lambda_A$ into $M^\lambda_B$
with image $\lambda_B(A)''$
(where $A \subseteq M^\lambda_A$ by identification with $\lambda_A(A)$).
\item For any $a'' \in M^\lambda_A$, $a''\Omega = \iota(a'')\Omega$,
and in particular, $M^\lambda_A \Omega \subseteq M^\lambda_B \Omega$.
\end{enumerate}
\end{proposition}

\begin{proof} (1) The subspace $H_A$ of $H_B$ is stable for $\lambda_B(A)$ and hence also for its double commutant. Since $\Omega \in H_A$ is
separating for $M^\lambda_B$ and hence also for $\lambda_B(A)''$,
it follows that the map of compression to $H_A$  is an injective 
$*$-homomorphism from $\lambda_B(A)''$ with image contained 
in $\lambda_A(A)'' = M^\lambda_A$. This map 
is clearly strongly continuous and so by Lemma \ref{normality}
it is normal and its image which is a von Neumann algebra containing
$\lambda_A(A)$ must
be $M^\lambda_A$. Just let $\iota$ be the inverse map.\\
(2) Since the compression to $H_A$ of $\iota(a'')$
is $a''$, $\Omega \in H_A$ and $H_A$ is stable for
$\lambda_B(A)''$, this is immediate.
\end{proof}

\begin{remark}
It can be verified that $\iota$ is the map defined by 
 $\iota(x)(b\Omega) = J_B(b^{*}x^* \Omega)$
for $x \in M^\lambda_A$ and $b \in B$.
\end{remark}

Applying Proposition \ref{comppair} to the tower $F_0(P) 
\subseteq F_1(P) \subseteq F_2(P) \subseteq \cdots$, each
pair of successive terms of which is a compatible pair
of pre-von Neumann algebras, 
we finally have a tower $M^\lambda_0 \subseteq M^\lambda_1 \subseteq \cdots$
of finite von Neumann algebras. 
Nevertheless we continue to regard $M^\lambda_k$ as a subset of
${\mathcal L}(H_k)$ and note that the $M^\lambda_k $ have a common cyclic and separating
vector $H_0 \ni \Omega = \Omega_1 = \Omega_2 = \cdots$.

\section{Factoriality and relative commutants}

In this section we show that the tower $M^\lambda_0 \subseteq
M^\lambda_1 \subseteq \cdots$ of finite von Neumann algebras constructed
in Section 4 is in fact a tower of $II_1$-factors, and more
generally that $(M^\lambda_0)' \cap M^\lambda_k$ is
$P_k \subseteq F_k(P) \subseteq M^\lambda_k$. We begin by justifying the words `more generally' of the previous sentence.
Throughout this section, we fix a $k \geq 0$.

\begin{lemma}\label{moregen}
Suppose that $(M^\lambda_0)' \cap M^\lambda_k$ is
$P_k \subseteq F_k(P) \subseteq M^\lambda_k$. Then, for $1 \leq i \leq k$, the relative commutant  $(M^\lambda_i)' \cap M^\lambda_k
= P_{i,k} \subseteq P_k$ (where $P_{i,k}$ is as in the
last paragraph of Section 2). In particular,
$M^\lambda_k$ is a factor.
\end{lemma}

\begin{proof} 
A straightforward pictorial argument shows that $P_{i,k} \subseteq
P_k \subseteq F_k(P)$ commutes with (the image in $F_k(P)$ of) $F_i(P)$ and hence also with (the image in $M^\lambda_k$ of) $M^\lambda_i$ (which is its double commutant by Proposition \ref{comppair}); 
it now suffices to see than an element of $P_k \subseteq F_k(P)$
that commutes with (the image in $F_k(P)$ of) $F_i(P)$ is necessarily
in $P_{i,k}$. Take such an element, say $x \in P_k$ and
consider the element of $P_{k+i} \subseteq F_i(P)$ shown
on the left in Figure \ref{eltofpkpli} which is seen to map
(under the inclusion map) to the element of $P_{2k} \subseteq
F_k(P)$ shown on the right.

\begin{figure}[!htb]
\psfrag{A}{$\mapsto$}
\psfrag{i}{\tiny $i$}
\psfrag{k-i}{\tiny $k\!-\!i$}
\includegraphics[height=1.5cm]{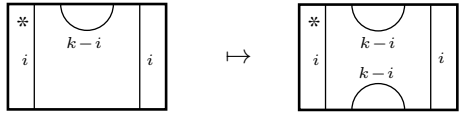}
\caption{An element of $P_{k+i} \subseteq F_i(P)$ and its image in 
$P_{2k} \subseteq
F_k(P)$}
\label{eltofpkpli}
\end{figure}
The condition that this element commutes with $x$ in $F_k(P)$ is easily
seen to translate to the tangle equation:
\begin{figure}[!htb]
\psfrag{x}{$x$}
\psfrag{i}{\tiny $i$}
\psfrag{k-i}{\tiny $k\!-\!i$}
\includegraphics[height=2.5cm]{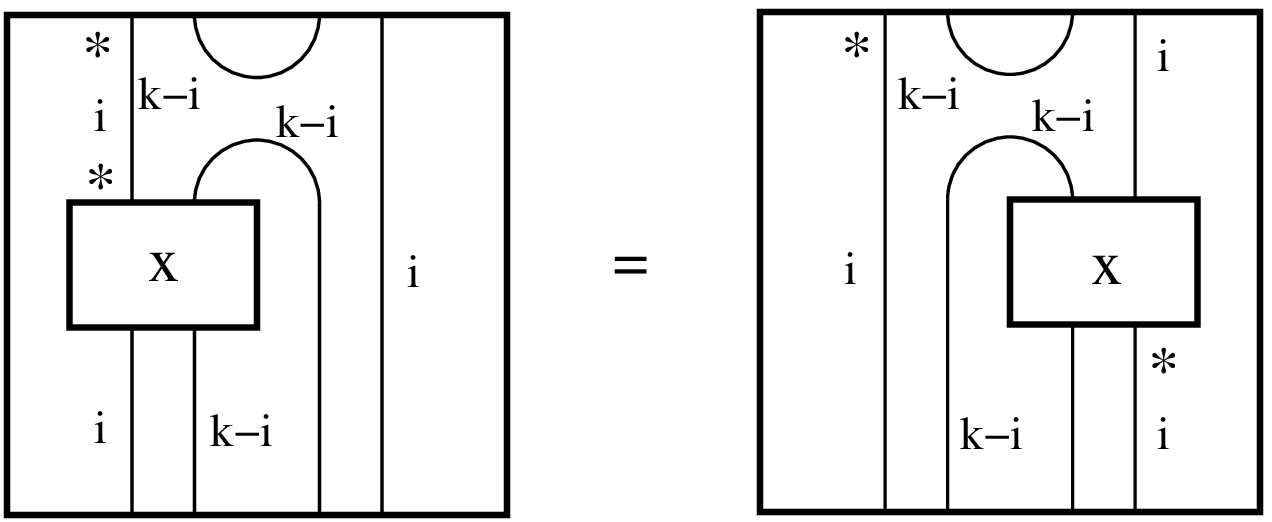}
\label{eqnforx}
\end{figure}

\noindent
which holds in $P_{2k}$. But now, taking the conditional expectation
of both sides into $P_k$ shows that $\delta^i x = Z_{EL(i)^k_k}(x)$, 
whence $x \in P_{i,k}$.
\end{proof}

Verification of the hypothesis of Lemma \ref{moregen} is computationally
involved and is the main result of this section which we state
as a proposition. The bulk of the work in proving this proposition
is contained in establishing Propositions \ref{ccomm} and \ref{dcomm}. In the course of proving this proposition, the
following notation and fact will be used. For $a \in F_k(P)$, define $[a] = \{\xi  \in H_k :
\lambda_k(a)(\xi) = \rho_k(a)(\xi)\}$, which is a closed
subspace of $H_k$. Now
observe that since $\Omega$ is separating for $M^\lambda_k$, for $x \in M^\lambda_k$, the operator
equation $ax = xa$ is equivalent to the condition $x\Omega \in [a]$.

\begin{proposition}\label{relcomm}
$(M^\lambda_0)' \cap M^\lambda_k = P_k \subseteq F_k(P) \subseteq
M^\lambda_k$.
\end{proposition}

\begin{proof}
As in the proof of Lemma \ref{moregen}, an easy pictorial calculation shows that $P_k \subseteq F_k(P)$
certainly commutes with all elements of (the image in $F_k(P)$ of)
$F_0(P)$ and therefore also with (the image in $M^\lambda_k$ of) $M^\lambda_0$ (which is its double commutant by Proposition \ref{comppair}).

To verify the other containment, we will show that any element
of $M^\lambda_k$ that commutes with the specific elements $c,d \in F_0(P)$
shown in Figure \ref{eltc}
\begin{figure}[!htb]
\psfrag{=}{$\mapsto$}
\psfrag{k}{\tiny $k$}
\psfrag{A}{$F_0(P) \supseteq P_1 \ni c=$}
\psfrag{B}{$\in P_{k+1} \subseteq F_k(P)$}
\psfrag{C}{$F_0(P) \supseteq P_2 \ni d=$}
\psfrag{D}{$\in P_{k+2} \subseteq F_k(P)$}
\includegraphics[height=2.2cm]{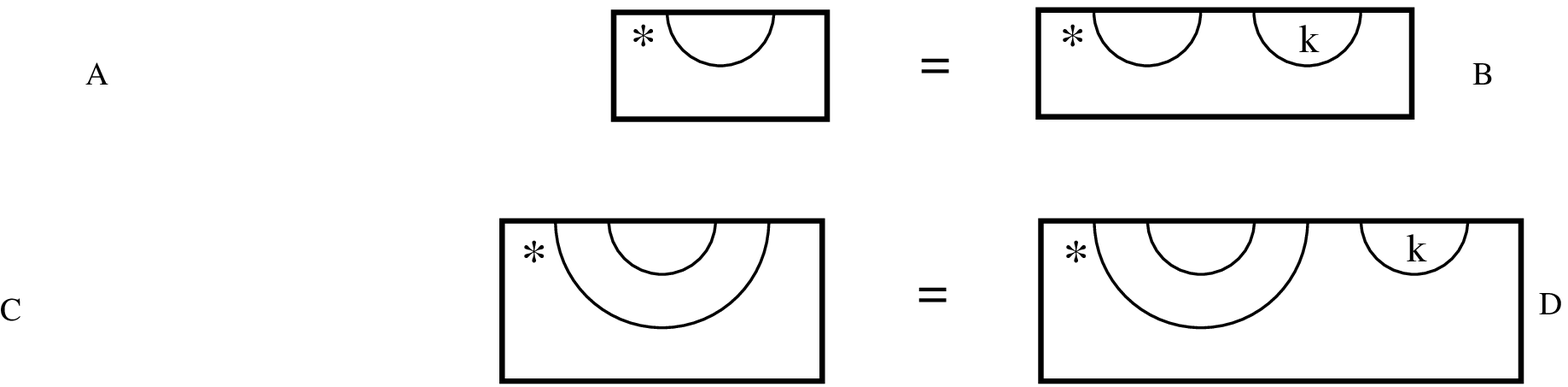}
\caption{The elements $c,d \in F_0(P)$ and their images in $F_k(P)$}
\label{eltc}
\end{figure}
\noindent
is necessarily in $P_k \subseteq F_k(P)$. We remark that the elements $c$ and $d$ also play a prominent role in the proof of
\cite{GnnJnsShl}.

Suppose now that $x \in M^\lambda_k$ commutes with both $c,d \in
F_k(P)$,
or equivalently that $x\Omega = (x_k,x_{k+1},\cdots) \in [c] \cap [d]$. It then follows from
Propositions \ref{ccomm} and \ref{dcomm} that $x\Omega = x_k\Omega$
or that $x = x_k \in P_k$.
\end{proof}

Proposition \ref{ccomm} identifies $[c]$ as a certain explicitly defined (closed) subspace of
$C_k \subseteq H_k$. Here
$C_k = \oplus_{n=k}^\infty C^n_k$ where $C^n_k \subseteq P_n$
is defined as the range of $Z_{X^n_k}$ where $X^n_k$, defined for $n \geq k$, is the annular tangle in Figure \ref{eltqn}, which has $n-k$ cups. We will have occasion
to use the following observation.
\begin{figure}[!htb]
\psfrag{A}{$\cdots$}
\psfrag{2m}{\tiny $2k$}
\includegraphics[height=1.5cm]{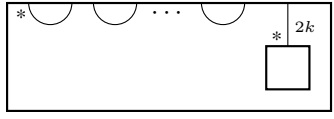}
\caption{The tangle $X^n_k$}
\label{eltqn}
\end{figure}

\begin{remark}\label{qnkperp}
We shall write $(C^n_k)^\perp$ to denote the orthogonal complement of
$C^n_k$ in $P_n$. It is a consequence of the faithfulness of $\tau$ that
$x \in C^n_k$ exactly when it satisfies the capping condition of Figure \ref{kqnkk}.
\begin{figure}[!htb]
\psfrag{A}{$\cdots$}
\psfrag{2m}{\tiny $2k$}
\psfrag{n-k}{\tiny $n\!-\!k$}
\psfrag{x}{$x$}
\psfrag{b}{$=\ 0$}
\includegraphics[height=1.5cm]{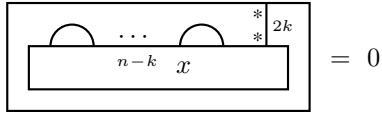}
\caption{The capping condition}
\label{kqnkk}
\end{figure}
\end{remark}

\begin{proposition}\label{ccomm} $[c] = C_k.$
\end{proposition}

We will take up the proof of Proposition \ref{ccomm} after that of Proposition \ref{dcomm}.

\begin{proposition}\label{dcomm} $C_k \cap [d] = P_k \subseteq H_k.$
\end{proposition}

\begin{proof}
Suppose that $\xi = (x_k,x_{k+1},\cdots) \in C_k$, so that for $n \geq k$,  there exist $y^n_k \in P_k$ such that  $x_n = Z_{X^n_k}(y^n_k)$.
We then compute
\begin{eqnarray*}
\lambda_k(d)(\xi) - \rho_k(d)(\xi) &=& \sum_{n=k}^\infty (d \# x_n -x_n \# d)\\
&=& \sum_{n=k}^\infty \sum_{t=|n-(k+2)|+k}^{n+2} (d \# x_n -x_n \# d)_t\\
&=& \sum_{t=k}^\infty \sum_{n=|t-(k+2)|+k}^{t+2} [d, Z_{X^n_k}(y^n_k)]_t
\end{eqnarray*}

Notice that if $t > k+2$, then the $P_t$ component
of $\lambda_k(d)(\xi) - \rho_k(d)(\xi)$
is given by
$$
\sum_{n=t-2}^{t-1} [d, Z_{X^n_k}(y^n_k)]_t
$$
since the other 3 commutators vanish, and is consequently given by
  $Z_{Y^t_k}(y^{t-1}_k + y^{t-2}_k) - Z_{Z^t_k}(y^{t-1}_k + y^{t-2}_k)$
\noindent
 where $Y^t_k$ and $Z^t_k$, defined for $t \geq k+2$, are the tangles in Figure \ref{eltqn2},
each having a double cup and $t-k-2$ single cups.
\begin{figure}[!htb]
\psfrag{A}{$\cdots$}
\psfrag{2m}{\tiny $2k$}
\psfrag{B}{$Y^t_k=$}
\psfrag{C}{$Z^t_k=$}
\includegraphics[height=1.5cm]{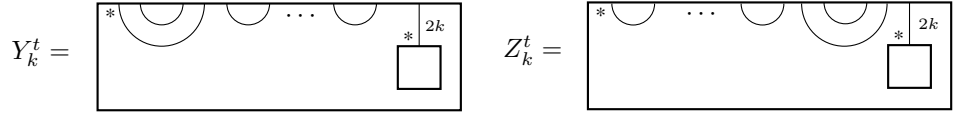}
\caption{The tangles $Y^t_k$ and $Z^t_k$}
\label{eltqn2}
\end{figure}
Thus if $\xi \in [d]$, then for $t > k+2$, $Z_{Y^t_k}(y^{t-1}_k + y^{t-2}_k) = Z_{Z^t_k}(y^{t-1}_k + y^{t-2}_k)$.

Note now that if there are at least 2 single cups, i.e., if $t \geq k+4$,
then, capping off the points $\{4,5\}$ of $Y^t_k$ and $Z^t_k$
gives $Y^{t-1}_k$ and $Z^{t-1}_k$. It follows by induction that
for $t \geq k+3$, $Z_{Y^{k+3}_k}(y^{t-1}_k + y^{t-2}_k) = Z_{Z^{k+3}_k}(y^{t-1}_k + y^{t-2}_k)$. But now, capping off $\{1,2\}$, $\{3,6\}$ and $\{4,5\}$
of $Y^{k+3}_k$ and $Z^{k+3}_k$ gives $\delta$ times the identity
tangle $I^k_k$ for $Y^{k+3}_k$ and $\delta^3$ times $I^k_k$
for $Z^{k+3}_k$ for $Z^{k+3}_k$. Since $\delta >1$, 
it follows that
for $t \geq k+3$, $y^{t-1}_k + y^{t-2}_k = 0$. Setting $y_{k+1} = y$,
we have $y_{k+3} = y_{k+5} = \cdots = y = -y_{k+2} = -y_{k+4} = \cdots$.

Finally, since $x_n = Z_{X^n_k}(y^n_k)$, we see that $\tau(x_n^*x_n) = \tau((y^n_k)^*y^n_k) = \tau(y^*y)$ for $n \geq k+1$. Hence
$||\xi||^2_{H_k} = \tau(x_k^*x_k) + \tau(y^*y) ( \delta +
\delta^2 + \cdots )$. Since this is to be finite (and $\delta>1$) , it follows that $y = 0$
or equivalently that $x_n = 0$ for $n \geq k+1$,
and hence $\xi \in P_k \subseteq H_k$.
\end{proof}

The proof of Proposition \ref{ccomm} involves analysis of linear
equations involving a certain class of annular tangles that we
will now define. Given $m,n \geq k \geq 0$ and subsets $A \subseteq
\{1,2,\cdots,m-k\}$ and $B \subseteq \{1,2,\cdots,n-k\}$ of equal
cardinality we define an annular tangle $T(k,A,B)^m_n$ by the following three requirements:  (i) there is a strictly increasing bijection
$f_{AB} : A \rightarrow B$ such that for $\alpha \in A$, the marked points on the external box that are
labelled by $2\alpha-1$ and $2\alpha$ are joined to the points on the internal box labelled by $2f_{AB}(\alpha)-1$ and $2f_{AB}(\alpha)$, (ii) the last $2k$ points on the external box are joined with the last $2k$
points on the internal box, and (iii) the rest of the marked points on the internal and external
boxes are capped off in pairs without nesting by joining each odd point
to the next even point.  
(In all cases of interest, the sets $A$ and $B$ will be intervals of positive integers; for
example we write $A = [3,7]$ to mean $A = \{3,4,5,6,7\}$.)
We illustrate with an example
of the tangle $T(1,[4,5],[3,4])^8_5$ in Figure \ref{segmented}.
\begin{figure}[!htb]
\includegraphics[height=3cm]{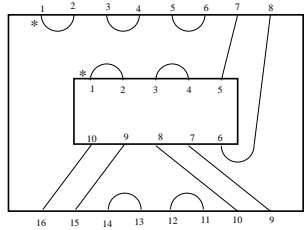}
\caption{The tangle $T(1,[4,5],[3,4])^8_5$}
\label{segmented}
\end{figure}

We leave it to the reader to verify that
the composition formula for this class of tangles is given by:
\begin{equation*}
Z_{T(k,A,B)^m_n} \circ Z_{T(k,C,D)^n_p} = \delta^{n - k - |B \cup C|} 
Z_{T(k,E,F)^m_p},
\end{equation*}
where $E = f_{AB}^{-1}(B \cap C)$ and $F = 
f_{CD}(B \cap C)$.

We now isolate a key lemma used in the proof of Proposition 
\ref{ccomm}.

\begin{lemma}\label{ccommlem}
For $n \geq k$, the map $(C^n_k)^\perp \ni x \mapsto z = (c\#x - x \# c)_{n+1} \in P_{n+1}$
is injective with inverse
given by 
$$
x = \sum_{t=1}^{n-k} \delta^{-t} Z_{T(k,[1,n+1-t-k],[t+1,n-k+1])^n_{n+1}}(z).
$$
\end{lemma}

\begin{proof}
Pictorially, $z$ is given in terms of $x$ as in Figure \ref{zis}.
\begin{figure}[!htb]
\psfrag{$x_n$}{$x$}
\psfrag{z}{$z$}
\psfrag{=}{$=$}
\psfrag{z =}{$=$}
\psfrag{2k}{\tiny $2k$}
\psfrag{2(n-k)}{\tiny $2(n\!-\!k)$}
\psfrag{2(n+1-k)}{\tiny $2(n\!+\!1\!-\!k)$}
\includegraphics[height=1.5cm]{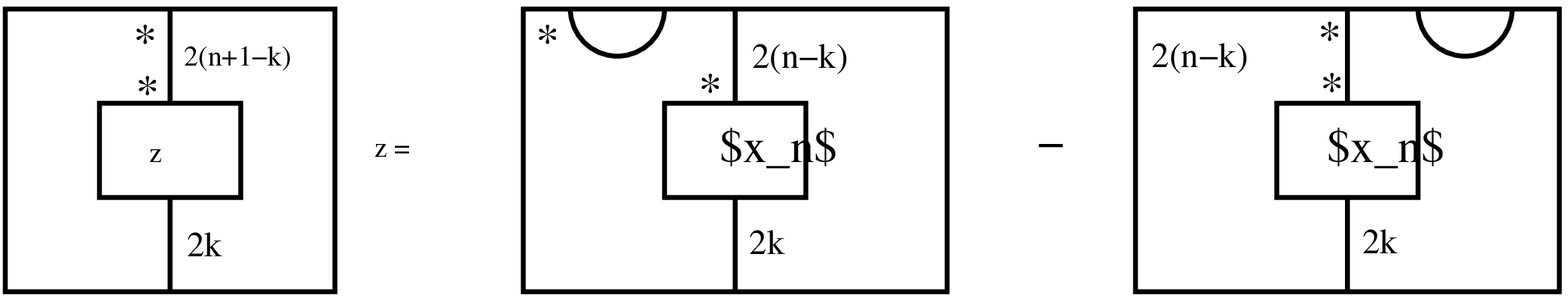}
\caption{Pictorial expression for $z$ in terms of $x$}
\label{zis}
\end{figure}
\noindent
It then follows by applying appropriate annular tangles that for any $t$ between $1$ and $n-k$, the pictorial equation of
Figure \ref{ztis} holds
\begin{figure}[!htb]
\psfrag{=}{$=$}
\psfrag{-}{$-$}
\psfrag{n-k-t}{\tiny $n-k-t$}
\psfrag{n-k-t+1}{\tiny $n-k-t+1$}
\psfrag{t}{\tiny $t$}
\psfrag{t-1}{\tiny $t-1$}
\psfrag{A}{$\cdots$}
\psfrag{$x_n$}{$x$}
\psfrag{$z$}{$z$}
\psfrag{2(n-k+1)}{\tiny $2(n\!-\!k\!+\!1)$}
\psfrag{2(n-k)}{\tiny $2(n\!-\!k
)$}
\psfrag{2k}{\tiny $2k$}
\includegraphics[height=3.15cm]{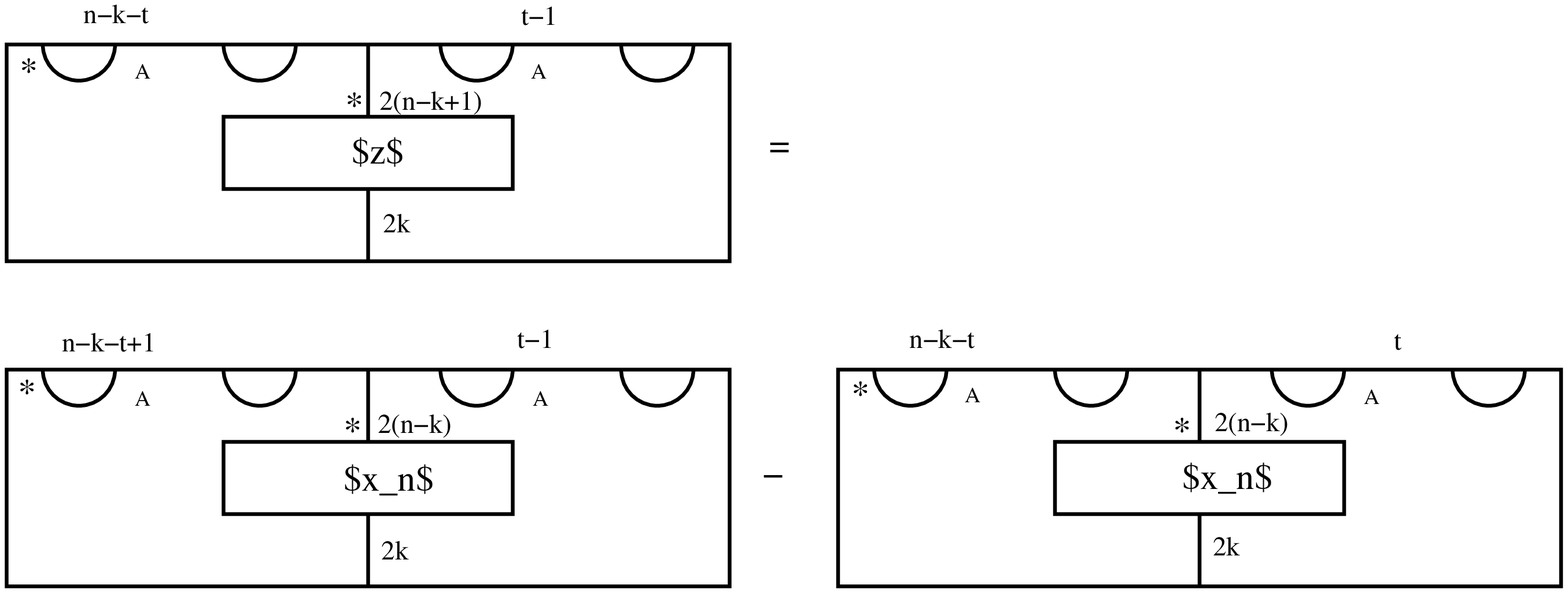}
\caption{Consequence of Figure \ref{zis}}
\label{ztis}
\end{figure}
(where the numbers on top of the ellipses indicate the number of cups).
Sum over $t$ between $1$ and $n-k$ and use the patent telescoping on the
right  to get the equation of Figure \ref{z3is}.
\begin{figure}[!htb]
\psfrag{n-k-t}{\tiny $n\!-\!k\!-\!t$}
\psfrag{n-k-t+1}{\tiny $n-k$}
\psfrag{t}{\tiny $n-k$}
\psfrag{t-1}{\tiny $t\!-\!1$}
\psfrag{S}{$\sum_{t=1}^{n-k}$}
\psfrag{A}{$\cdots$}
\psfrag{$x_n$}{$x$}
\psfrag{$z$}{$z$}
\psfrag{n-k}{\tiny $n\!-\!k$}
\psfrag{2(n-k+1)}{\tiny $2(n\!-\!k\!+\!1)$}
\psfrag{2(n-k)}{\tiny $2(n\!-\!k)$}
\psfrag{2k}{\tiny $2k$}
\includegraphics[width=12cm]{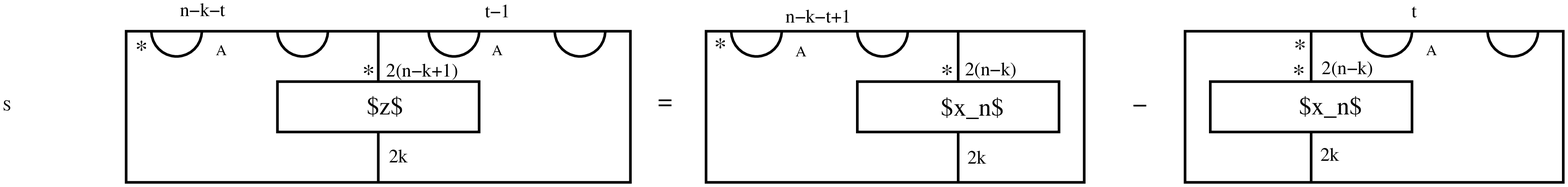}
\caption{Consequence of Figure \ref{ztis}}
\label{z3is}
\end{figure}

Finally, cap off the pairs of points $\{1,2\},\{3,4\},\cdots,\{2(n-k)-1,2(n-k)\}$ on all the tangles in the Figure \ref{z3is} to conclude that 
\begin{equation*}
\sum_{t=1}^{n-k} \delta^{n-k-t} Z_{T(k,[1,n+1-t-k],[t+1,n-k+1])^n_{n+1}}(z) = \delta^{n-k} x,
\end{equation*}
because the second term on the right vanishes by Remark \ref{qnkperp}.
\end{proof}

\begin{corollary}\label{xnxm}
Suppose that $\xi = (x_k,x_{k+1},\cdots) \in \oplus_{n=k}^\infty (C^n_k)^{\perp} =  C_k^{\perp}$ and satisfies $\lambda_k(c)(\xi) = \rho_k(c)(\xi)$. Then, for $m > n > k$ with $m-n = 2d$, we have:
\begin{eqnarray*}
x_n &=& \sum_{t=1}^{n-k} \delta^{-(t+d-1)}  \left\{Z_{T(k,[1,n+1-t-k],[t+d,n-k+d])^{n}_{m}}(x_m)\right.\\
 & &  - \left.Z_{T(k,[1,n+1-t-k],[t+d+1,n-k+d+1])^{n}_{m}}(x_m)\right\}
\end{eqnarray*}
\end{corollary}

\begin{proof} As in the proof of Proposition \ref{dcomm}, some calculation shows that for $n>k$, the $P_{n+1}$ component of $\lambda_k(c)(\xi) - \rho_k(c)(\xi)$  
is seen to be
$$
\sum_{s=n}^{n+2} [c,x_s]_{n+1}
$$
of which the middle term is seen to vanish.
Hence it 
is given by the difference of the right and left hand
sides of the equation in Figure \ref{ccommeqn}, and therefore if
$\lambda_k(c)(\xi) = \rho_k(c)(\xi)$, the equations of Figure \ref{ccommeqn}  hold for all $n>k$.
\begin{figure}[!htb]
\psfrag{$x_n$}{$x_n$}
\psfrag{$x_{n+2}$}{$x_{n+2}$}
\psfrag{2k}{\tiny $2k$}
\psfrag{2(n-k)}{\tiny $2(n\!-\!k)$}
\psfrag{2(n-k+1)}{\tiny $2(n\!-\!k\!+\!1)$}
\includegraphics[width=12cm]{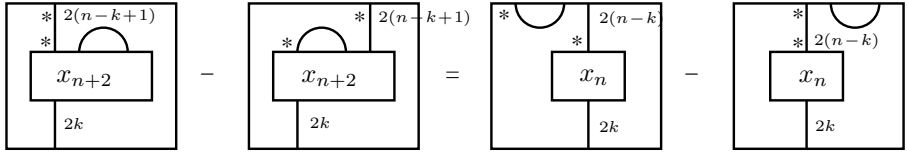}
\caption{The condition for commuting with $c$}
\label{ccommeqn}
\end{figure}

Let $z$ denote the value of the left and right hand sides of the equation in Figure \ref{ccommeqn}, so that on the one hand, we have
$$z = Z_{T(k,[1,n-k+1],[1,n-k+1])^{n+1}_{n+2}}(x_{n+2}) - Z_{T(k,[1,n-k+1],[2,n-k+2])^{n+1}_{n+2}}(x_{n+2}),$$ while on the other,
 $ z = (c\#x_n - x_n \# c)_{n+1}$. We now have by Lemma \ref{ccommlem} that
\begin{eqnarray*}
x_n &=& \sum_{t=1}^{n-k} \delta^{-t} Z_{T(k,[1,n+1-t-k],[t+1,n-k+1])^n_{n+1}}  \\
& & \left\{Z_{T(k,[1,n-k+1],[1,n-k+1])^{n+1}_{n+2}}(x_{n+2}) - Z_{T(k,[1,n-k+1],[2,n-k+2])^{n+1}_{n+2}}(x_{n+2})\right\} \\
&=& \sum_{t=1}^{n-k} \delta^{-t} \left\{Z_{T(k,[1,n+1-t-k],[t+1,n-k+1])^{n}_{n+2}}(x_{n+2})\right.  \\
 & &  - \left.Z_{T(k,[1,n+1-t-k],[t+2,n-k+2])^{n}_{n+2}}(x_{n+2})\right\},
\end{eqnarray*}
which proves the $d=1$ case. For $d > 1$, assume
by induction on $d$ that
\begin{eqnarray*}
x_{n+2} = \sum_{s=1}^{n+2-k} \delta^{-(s+d-2)} &\times&\left\{Z_{T(k,[1,n+3-s-k],[s+d-1,n-k+d+1])^{n+2}_{m}}(x_m)\right.\\
 & &  - \left.Z_{T(k,[1,n+3-s-k],[s+d,n-k+d+2])^{n+2}_{m}}(x_{m})\right\}
\end{eqnarray*}

Substituting the expression for $x_{n+2}$ from the last equation into the preceding equation,
we find:
\begin{eqnarray*}
\ x_n &=& \sum_{t=1}^{n-k} \sum_{s=1}^{n+2-k} \delta^{-(t+s+d-2)} \times \\
& & \left\{     Z_{T(k,[1,n+1-t-k],[t+1,n-k+1])^{n}_{n+2}} \circ   Z_{T(k,[1,n+3-s-k],[s+d-1,n-k+d+1])^{n+2}_{m}}(x_m)             \right.\\
& &- Z_{T(k,[1,n+1-t-k],[t+1,n-k+1])^{n}_{n+2}} \circ  Z_{T(k,[1,n+3-s-k],[s+d,n-k+d+2])^{n+2}_{m}}(x_{m})\\
& &- Z_{T(k,[1,n+1-t-k],[t+2,n-k+2])^{n}_{n+2}} \circ Z_{T(k,[1,n+3-s-k],[s+d-1,n-k+d+1])^{n+2}_{m}}(x_m)\\
& &+ \left. Z_{T(k,[1,n+1-t-k],[t+2,n-k+2])^{n}_{n+2}} \circ Z_{T(k,[1,n+3-s-k],[s+d,n-k+d+2])^{n+2}_{m}}(x_{m}) \right\}
\end{eqnarray*}

This is a sum over varying $t$ and $s$ of 4 terms, each with a multiplicative factor
of a power of $\delta$. Note that as $x_m \in (C^m_k)^\perp$, it follows from Remark \ref{qnkperp} that the first two terms vanish unless $t \leq n+2-s-k$ and the last two vanish unless $t \leq n+1-s-k$. So the sum of the first two terms may be taken over
those $(t,s)$ satisfying $t+s \leq n+2-k$ while the sum of the last two terms may
be taken over those $(t,s)$ satisfying $t+s \leq n+1-k$. In this range, the composition formula for the
$T(k,A,B)$ tangles gives:
\begin{eqnarray*}
 Z_{T(k,[1,n+1-t-k],[t+1,n-k+1])^{n}_{n+2}} &\circ& Z_{T(k,[1,n+3-s-k],[s+d-1,n-k+d+1])^{n+2}_{m}}\\  &=&\left\{ \begin{array}{ll} Z_{T(k,[1,n+1-t-k],[t+d,n-k+d])^{n}_{m}} {\rm if}~s=1 \\ \delta Z_{T(k,[1,n+3-s-k-t],[s+t+d-1,n-k+d+1])^{n}_{m}} {\rm if}~s>1\end{array} \right. \\
Z_{T(k,[1,n+1-t-k],[t+1,n-k+1])^{n}_{n+2}} &\circ&  Z_{T(k,[1,n+3-s-k],[s+d,n-k+d+2])^{n+2}_{m}} \\ &=& \left\{ \begin{array}{ll} Z_{T(k,[1,n+1-t-k],[t+d+1,n-k+d+1])^{n}_{m}} {\rm if}~s=1 \\ \delta Z_{T(k,[1,n+3-s-k-t],[s+t+d,n-k+d+2])^{n}_{m}} {\rm if}~s>1\end{array} \right. \\
Z_{T(k,[1,n+1-t-k],[t+2,n-k+2])^{n}_{n+2}} &\circ& Z_{T(k,[1,n+3-s-k],[s+d-1,n-k+d+1])^{n+2}_{m}}\\ &=& Z_{T(k,[1,n+2-s-k-t],[s+t+d,n-k+d+1])^{n}_{m}}\\
Z_{T(k,[1,n+1-t-k],[t+2,n-k+2])^{n}_{n+2}} &\circ& Z_{T(k,[1,n+3-s-k],[s+d,n-k+d+2])^{n+2}_{m}}\\  &=& Z_{T(k,[1,n+2-s-k-t],[s+t+d+1,n-k+d+2])^{n}_{m}}
\end{eqnarray*}

Finally, note that the first two terms with $s=1$ sum to exactly the desired
expression for $x_n$ while for $s >1$, the first term corresponding to $(t,s)$ cancels
against the third term corresponding to $(t,s-1)$ while the second term
corresponding to $(t,s)$ cancels with the fourth term corresponding to
$(t,s-1)$, finishing the induction.
\end{proof}

We only need one more thing in order to complete the proof of
Proposition \ref{ccomm}.
\begin{lemma}\label{easy}
Suppose that $x = x_n \in P_n \subseteq F_k(P)$ and let $y = Z_{T(k,A,B)^m_n}(x) \in P_m \subseteq F_k(P)$ for some annular tangle
$T(k,A,B)^m_n$. Then $||y||_{H_k} \leq \delta^{\frac{1}{2}(n+m) - (|A|+k)} ||x||_{H_k}$.
\end{lemma}

\begin{proof}
Since for $z \in P_u \subseteq F_k(P)$, the norm relation $\delta^{k-u} ||z||^2_{H_k}
= ||z||^2_{P_u} = \tau(z^*z)$ holds, the inequality we need to verify may be
equivalently stated as 
$||y||_{P_m}^2 \leq \delta^{2(n-|A|-k)}  ||x||_{P_n}^2$. We leave it to the reader
to see that this follows from Remark \ref{plalg}.
\end{proof}

\begin{proof}[Proof of Proposition \ref{ccomm}]
Observe first that $C^n_k$ is easily verified to be in $[c]$ and so
$C_k \subseteq [c]$. To show the other inclusion we will show that
$C_k^\perp \cap [c] = \{0\}$.

Take $\xi = (x_k,x_{k+1},\cdots) \in C_k^\perp = \oplus_{n=k}^\infty (C^n_k)^\perp$
and suppose that $\xi \in [c]$.
Then, $x_k =0$ since $C^k_k = P_k$ and $x_k \in (C_k^k)^\perp$.
Next, it follows from Corollary \ref{xnxm} and Lemma \ref{easy} that
if $m > n > k$ and $m-n = 2d$, then
$$
||x_n||_{H_k} \leq \sum_{t=1}^{n-k} \delta^{-(t+d-1)} \delta^{\frac{1}{2}(n+(n+2d))-(n+1-t)} 2 ||x_m||_{H_k} = 2(n-k)||x_m||_{H_k}.
$$
Since $\xi \in H_k$ clearly implies that $lim_{m \rightarrow \infty} ||x_m|| = 0$, it
follows that $x_n = 0$ also for $n >k$.
 Hence $\xi = 0$.
\end{proof}

We will refer to $M^\lambda_0 \subseteq M^\lambda_1$ as the
subfactor constructed from $P$.

\section{Basic construction tower, Jones projections and the main theorem}

In this section, we first show that the tower $M^\lambda_0 \subseteq
M^\lambda_1 \subseteq \cdots$ of $II_1$-factors may be identified
with the basic construction tower of  $M^\lambda_0 \subseteq
M^\lambda_1$ which is extremal and has index $\delta^2$. 
We also explicitly identify the Jones projections. Using this, we
easily prove our main theorem.

We begin with an omnibus proposition.

\begin{proposition}\label{omnibus}
For $k \geq 1$, the following statements hold.
\begin{enumerate} 
\item The trace preserving conditional expectation map
$E_{k-1}^\lambda: M^\lambda_k \rightarrow M^\lambda_{k-1}$ is given by
the restriction to $M^\lambda_k$ of the continuous extension
$H_k \rightarrow H_{k-1}$ of $E_{k-1}: F_k(P) \rightarrow F_{k-1}(P)$
(where $M^\lambda_k \subseteq H_k$ by $x \mapsto x\Omega$).
It is continuous for the strong operator topologies on the domain and
range.
\item The element $e_{k+1} \in
P_{k+1} \subseteq F_{k+1}(P)$ commutes with $F_{k-1}(P)$ and satisfies
$E^\lambda_k(e_{k+1}) = \delta^{-2}$ and
$e_{k+1}xe_{k+1} = E^\lambda_{k-1}(x)e_{k+1}$,
for any $x \in M^\lambda_k$.
\item The map $\theta_k : F_{k+1}(P) \rightarrow
End(F_k(P))$ defined by $\theta_k(a)(b) = \delta^2 E_k(abe_{k+1})$
is a homomorphism.
\item The tower $M^\lambda_{k-1} \subseteq
M^\lambda_k \subseteq M^\lambda_{k+1}$ of $II_1$-factors is (isomorphic to) a basic construction tower  with Jones projection given by $e_{k+1}$ and finite index $\delta^2$.
\item The subfactor $M^\lambda_{k-1} \subseteq M^\lambda_k$ is extremal.
\end{enumerate}
\end{proposition}

\begin{proof}
(1) Since Lemma \ref{single}(4) identifies $H_k$  with $L^2(M^\lambda_k,t^\lambda_k)$, it suffices to see that if $\tilde{e}_{k+1} \in {\mathcal L}(H_k)$ (for $k \geq 1$) denotes the orthogonal projection onto the
closed subspace $H_{k-1}$, then $\tilde{e}_{k+1}|_{F_k(P)} = E_{k-1}$.
Since $H_k = \oplus_{n=k}^\infty P_n$ while $H_{k-1}$ is the sum
of the subspaces $\oplus_{n=k}^\infty P_{n-1}$, it suffices to see that
the restriction of $E_{k-1}$ to $P_n \subseteq F_k(P)$ is the
orthogonal projection of $P_n$ onto $P_{n-1}$ included in $P_n$ via
the specification of Figure \ref{incldef}, which is easily verified. Continuity for the strong operator
topologies follows since $E_{k-1}^\lambda(x)$ is the compression
of $x$ to the subspace $H_{k-1}$.
\\
(2) We omit the simple calculations needed to verify this.\\ 
(3) Denote $\theta_k(a)(b)$ by $a.b$. What is to be seen then, is that
$(a\#b).c = a.(b.c)$ for $a,b \in F_{k+1}(P)$ and
$c \in F_k(P)$. 
Some computation with the definitions yields for $a \in P_m \subseteq F_{k+1}(P)$
and $b \in P_n \subseteq F_k(P)$, $a.b = \sum_{t=|m-n-1|+k}^{m+n-k-1} (a.b)_t$ where $(a.b)_t$ is given by the tangle in Figure \ref{adotbt}.
\begin{figure}[!htb]
\psfrag{k+1}{\tiny $k\!+\!1$}
\psfrag{k-1}{\tiny $k\!-\!1$}

\psfrag{a}{$a$}
\psfrag{b}{$b$}
\psfrag{m+p-n-s}{\tiny $m\!+\!p\!-\!n\!-\!s$}
\psfrag{m+p-n-s}{\tiny $m\!+\!p\!-\!n\!-\!s$}
\psfrag{n+s-m-p}{\tiny $n\!+\!s\!-\!m\!-\!p$}
\psfrag{t+p-s-k-1}{\tiny $t\!+\!p\!-\!s\!-\!k\!-\!1$}
\psfrag{m+n-k-t-1}{\tiny $m\!+\!n\!-\!t\!-\!k\!-\!1$}
\psfrag{n-m+t-k-1}{\tiny $n\!-\!m\!+\!t\!-\!k\!-\!1$}
\psfrag{p-t+s}{\tiny $p\!-\!t\!+\!s$}
\psfrag{n+t-m}{\tiny $n\!-\!m\!+\!t$}
\psfrag{m+t-n}{\tiny $m\!-\!n\!+\!t$}
\psfrag{-k+1}{\tiny $-\!k\!+\!1$}
\psfrag{-k-1}{\tiny $-\!k\!-\!1$}

\includegraphics[height=2.5cm]{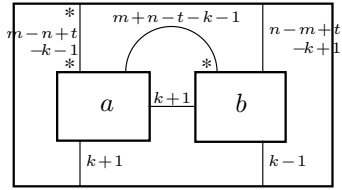}
\caption{Definition of the $P_t$ component of $a.b$.}
\label{adotbt}
\end{figure}
Thus for $a \in P_m \subseteq F_{k+1}(P),b \in P_n \subseteq F_{k+1}(P)$ and $c \in P_p \subseteq F_k(P)$, we have:

$$
(a\#b).c = \sum_{(t,s) \in I(m,n,p)} ((a\#b)_t.c)_s {\rm \ \ and\ \ }
a.(b.c) = \sum_{(v,u) \in J(m,n,p)} (a.(b.c)_v)_u,
$$
where $I(m,n,p) = \{(t,s) : |m-n|+k+1 \leq t \leq m+n-k-1, |t-p-1|+k \leq s \leq t+p-k-1\}$ and $J(m,n,p) = \{(v,u) : |n-p-1|+k \leq v \leq n+p-k-1, |m-v-1|+k \leq u \leq m+v-k-1\}$. We leave it to the reader to check that
(a) $I(m,n,p) = J(p+1,n,m-1)$ and
(b) the map  $T(m,n,p): I(m,n,p) \rightarrow J(m,n,p)$ defined by
$(t,s) \mapsto (v,u) = (max\{m+p,n+s\}-t,s)$ is a well-defined bijection 
with inverse $T(p+1,n,m-1)$
such that
(c) if $(t,s) \mapsto (v,u)$ then both $((a\#b)_t.c)_s$ and $(a.(b.c)_v)_u$
are equal to the figure on the left or on the
right in Figure \ref{action}
\begin{figure}[!htb]
\psfrag{a}{$a$}
\psfrag{b}{$b$}
\psfrag{c}{$c$}
\psfrag{k+1}{\tiny $k\!+\!1$}
\psfrag{k-1}{\tiny $k\!-\!1$}
\psfrag{m+p-n-s}{\tiny $m\!+\!p\!-\!n\!-\!s$}
\psfrag{m+p-n-s}{\tiny $m\!+\!p\!-\!n\!-\!s$}
\psfrag{n+s-m-p}{\tiny $n\!+\!s\!-\!m\!-\!p$}
\psfrag{t+p-s-k-1}{\tiny $t\!+\!p\!-\!s\!-\!k\!-\!1$}
\psfrag{m+n-t-k-1}{\tiny $m\!+\!n\!-\!t\!-\!k\!-\!1$}
\psfrag{n-m+t-k-1}{\tiny $n\!-\!m\!+\!t\!-\!k\!-\!1$}
\psfrag{p-t+s}{\tiny $p\!-\!t\!+\!s$}
\psfrag{s-p+t}{\tiny $s\!-\!p\!+\!t$}
\psfrag{m-n+t}{\tiny $m\!-\!n\!+\!t$}
\psfrag{-k+1}{\tiny $-\!k\!+\!1$}
\psfrag{-k-1}{\tiny $-\!k\!-\!1$}
\includegraphics[height=2.7cm]{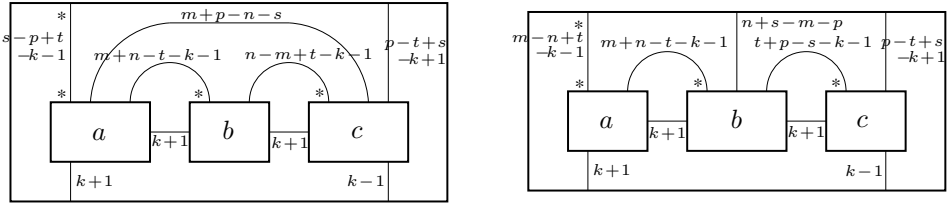}
\caption{$((a\#b)_t.c)_s  = (a.(b.c)_v)_u$.}
\label{action}
\end{figure}
according as $m+p \geq n+s$ or $m+p \leq n+s$.\\
(4) Extend  $\theta_k$ to a map, also denoted $\theta_k : M^\lambda_{k+1} \rightarrow
{\mathcal L}(H_k)$ by defining it on the dense subspace $M^\lambda_k \Omega \subseteq 
H_k$ by the formula (from Lemma 4.3.1
of \cite{JnsSnd}) $\theta_k(x)(y\Omega) = \delta^2 E^\lambda_k(xye_{k+1})$ for $x \in M^\lambda_{k+1}$ and $y \in M^\lambda_{k}$
and observing that it extends continuously to $H_k$.
It is easy to check that $\theta_k$ is $*$-preserving and it follows from the strong continuity of $E^\lambda_{k}$ that
$\theta_k$ preserves norm bounded strongly convergent nets.
Further, by (3), restricted to $F_{k+1}(P)$,
the map $\theta_k$ is a unital $*$-homomorphism. Now Kaplansky density, the convergence preserving property of $\theta_k$  and the strong continuity of multiplication
on norm bounded subsets imply
that $\theta_k$ is a unital $*$-homomorphism. Since its domain
is a factor, $\theta_k$ is also injective. By Lemma \ref{normality},
$\theta_k$ is normal so that its image is
a von Neumann subalgebra of ${\mathcal L}(H_k)$.

That this image  contains
$M^\lambda_k = \theta_k(M^\lambda_k)$ and $\tilde{e}_{k+1} 
= \theta_k(e_{k+1})$ is clear from the formula that defines $\theta_k$.
Therefore (with $J_k : H_k \rightarrow H_k$ denoting the modular conjugation operator
for $M^\lambda_k$) $J_k.im(\theta_k).J_k$ contains $M^\rho_k$
and $\tilde{e}_{k+1}$. Direct calculation also shows that for $x \in M^\lambda_{k+1}$ and $y \in M^\lambda_k$, $J_k\theta_k(x^*)J_k(y\Omega) = \delta^2E^\lambda_k(e_{k+1}yx)$ and consequently that
$J_k.im(\theta_k).J_k
\subseteq \lambda_k(M^\lambda_{k-1})^\prime$. The opposite
inclusion is equivalent to the statement that $J_k.im(\theta_k)^\prime.J_k
\subseteq \lambda_k(M^\lambda_{k-1})$. But then, $J_k.im(\theta_k)^\prime.J_k
\subseteq M^\lambda_k \cap (\tilde{e}_{k+1})^\prime$ which is easily
verified to be $\lambda_k(M^\lambda_{k-1})$ since $\tilde{e}_{k+1}$ implements
the conditional expectation of $M^\lambda_k$ onto $M^\lambda_{k-1}$.
We conclude that $im(\theta_k)$ is $J_k.\lambda_k(M^\lambda_{k-1})^\prime .J_k$.

So the $II_1$-factor $M^\lambda_{k+1}$ is isomorphic (via $\theta_k$)
to the basic construction of $M^\lambda_{k-1} \subseteq M^\lambda_k$ with the Jones projection being identified with $e_{k+1} \in M^\lambda_{k+1}$.
It also follows that the index is finite and equals $E^\lambda_k(e_{k+1})^{-1} = \delta^2.$\\
(5) By (4) it suffices to see that $M^\lambda_0 \subseteq M^\lambda_1$ is extremal for which, according to Corollary 4.5 of \cite{PmsPpa}, it is enough  that the anti-isomorphism $x \mapsto Jx^*J$
from $(M^\lambda_0)^\prime \cap M^\lambda_1$ to $M^\lambda_2 \cap (M^\lambda_1)^\prime$ be trace preserving, for the trace $t^\lambda_1$ on the left and the trace $t^\lambda_2$ on the right.
Here, all algebras involved are identified with subalgebras of ${\mathcal L}(
H_1)$ via the appropriate isomorphisms - $\iota$ for
$M^\lambda_0$ and $\theta_1$ for $M^\lambda_2$.
By Lemma~\ref{moregen},  $(M^\lambda_0)^\prime \cap M^\lambda_1$ is identified with $P_1 \subseteq
F_1(P) \subseteq {\mathcal L}(H_1)$ and $M^\lambda_2 \cap (M^\lambda_1)^\prime$ with $P_{1,2} \subseteq P_2 \subseteq F_2(P)$.

Working with the definition of $\theta_1$ shows that if $x \in P_1$ and $z
\in P_{1,2}$ are as in Figure \ref{obvious} and $y \in F_1(P)$ is arbitrary,
then $\theta_1(z)(y\Omega) = yx\Omega = Jx^*J(y\Omega)$.
Chasing through the identifications, this establishes that
that the anti-isomorphism $P_1 \rightarrow P_{1,2}$
given by $x \mapsto Jx^*J$ is the obvious one in Figure~\ref{obvious}.
\begin{figure}[!htb]
\psfrag{x}{$x$}
\psfrag{z =}{$z=$}
\psfrag{A}{$\mapsto$}
\includegraphics[height=1.5cm]{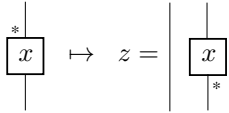}
\caption{The map from $P_1$ to $P_{1,2}$}
\label{obvious}
\end{figure}
Finally, noting that the traces restricted to $P_1$ and $P_{1,2}$ are both
$\tau$, sphericality of the subfactor planar algebra $P$ proves
the extremality desired.
\end{proof}

We now have all the pieces in place to prove our main theorem
which is exactly that
of \cite{GnnJnsShl}.

\begin{theorem}
Let $P$ be a subfactor planar algebra. The subfactor  $M^\lambda_0 \subseteq
M^\lambda_1$ constructed from $P$ is a
finite index and extremal subfactor with planar algebra isomorphic
to $P$.
\end{theorem}

\begin{proof}
By Proposition \ref{omnibus} and the remarks preceding it, the
subfactor  $M^\lambda_0 \subseteq
M^\lambda_1$ constructed from $P$ is extremal, of index $\delta^2$ (where $\delta$ is the modulus of $P$) and has 
basic construction tower given by $M^\lambda_0 \subseteq M^\lambda_1 \subseteq M^\lambda_2 \subseteq \cdots$.

Further, the towers
of relative commutants $(M^\lambda_1)^\prime \cap M^\lambda_k 
\subseteq (M^\lambda_0)^\prime \cap M^\lambda_k$
are identified, by Proposition \ref{relcomm}, with those of the $P_{1,k} \subseteq P_k$
with inclusions (as $k$ increases) given by the inclusion tangle, by the remarks preceding
Proposition \ref{filtalg}.

Under this identification, the Jones projections $e_{k+1} \in P_{k+1}$
are the Jones projections in the basic construction tower
$M^\lambda_0 \subseteq M^\lambda_1 \subseteq M^\lambda_2 \subseteq \cdots$.
Also the trace $t^\lambda_k$ restricted to $P_k$ is just its usual
trace $\tau$. It follows that the trace preserving conditional
expectation maps $(M^\lambda_0)^\prime \cap M^\lambda_k \rightarrow (M^\lambda_0)^\prime \cap M^\lambda_{k-1}$
and $(M^\lambda_0)^\prime \cap M^\lambda_k \rightarrow (M^\lambda_1)^\prime \cap M^\lambda_{k}$ agree with those of
the planar algebra $P$.

An appeal now to Jones' theorem (Theorem \ref{jones}) completes
the proof.
\end{proof}

\section{Agreement with the Guionnet-Jones-Shlyakhtenko model}

In this short section, we construct (without proofs) explicit trace preserving $*$-isomorphisms from the
graded algebras $Gr_k(P)$ of \cite{GnnJnsShl} to the filtered
algebras $F_k(P)$ and their inverse isomorphisms. Since we have not unravelled the inclusions
of $Gr_k(P)$ from \cite{GnnJnsShl}, we are not able to claim
the towers are isomorphic, but suspect that this is indeed the case.

For $k \geq 0$, the algebra $Gr_k(P)$ is defined to be the
graded algebra with underlying vector space given by
$\oplus_{n=k}^\infty P_n$ and multiplication defined as follows. For
$a \in P_m \subseteq Gr_k(P)$ and $b \in P_n \subseteq Gr_k(P)$, define
$a \bullet b \in P_{m+n-k} \subseteq Gr_k(P)$ by
$a \bullet b = (a \# b)_{m+n-k}$ - the highest (i.e., $t = m+n-k$) degree
component of $a \# b$ (and observe therefore that $Gr_k(P)$
is the associated graded algebra of the filtered algebra $F_k(P)$).

There is a $*$-algebra
structure on $Gr_k(P)$ defined exactly as in $F_k(P)$ which
also we will denote by $\dagger$.
There is also a trace $Tr_k$ defined on $Gr_k(P)$ as follows.
For $a \in P_m \subseteq Gr_k(P)$, define $Tr_k(a)$ by the tangle in Figure
\ref{trdef}
\begin{figure}[!htb]
\psfrag{a}{$a$}
\psfrag{m-k}{\tiny $m\!-\!k$}
\psfrag{T_{m-k}}{$T_{m-k}$}
\psfrag{k}{\tiny $k$}
\includegraphics[height=2.5cm]{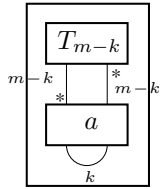}
\caption{Definition of $Tr_k(a)$.}
\label{trdef}
\end{figure}
\noindent
where $T_m \in P_m$ is defined to be the sum of all the Temperley-Lieb elements of $P_m$.

Define maps $\phi(k): Gr_k(P) \rightarrow F_k(P)$ and
$\psi(k): F_k(P) \rightarrow Gr_k(P)$ as follows.
For $i,j \geq k$, denoting the $P_i$ component of $\phi(k)|_{P_j}$ (resp. $\psi(k)|_{P_j}$) by
$\phi(k)^i_j$ (resp. $\psi(k)^i_j$), define $\phi(k)^i_j$ to be the sum of the maps given by all the
 `$k$-good $(j,i)$-annular tangles'  and $\psi(k)^i_j$ to be
$(-1)^{i+j}$ times the sum of the maps given by all the `$k$-excellent $(j,i)$-annular tangles', where,
for $k \leq i \leq j$, a  $(j,i)$-annular tangle is said to be $k$-good if (i) it has $2i$ through
strands, (ii) the $*$-regions of its internal and external boxes coincide
and (iii) the last $2k$ points
of the internal box are connected to the last $2k$ points of the external
box, and $k$-excellent if further, (iv) there is no nesting
among the strands connecting points on its internal box. 
Note that the matrices of $\phi(k)$ and $\psi(k)$, regarded
as maps from $\oplus_{n=k}^\infty P_n$ to itself, are block
upper triangular with diagonal blocks being identity matrices,
and consequently clearly invertible.

We then have the following result.

\begin{proposition}
For each $k \geq 0$, the maps $\phi(k)$ and $\psi(k)$ are mutually inverse
$*$-isomorphisms which take the traces $Tr_k$ and $\delta^k t_k$
to each other.
\end{proposition}

\bibliographystyle{amsalpha}

\begin{thebibliography}{10}

\bibitem[GnnJnsShl]{GnnJnsShl}A. Guionnet, V. F. R. Jones and D. Shlyakhtenko,
\textit{Random matrices, free probability, planar algebras and subfactors}, arXiv:0712.2904v1 [math.OA].

\bibitem[Jns]{Jns} V. F. R. Jones, \textit{Planar algebras}, arXiv:9909027v1 [math.QA].

\bibitem[Jns2]{Jns2} V. F. R. Jones, \textit{The planar algebra of a
bipartite graph}, In {\it Knots in Hellas '98 (Delphi)}, Vol. 24 of {\it Ser. Knots Everything}, 94--117, World Scientific (2000).

\bibitem[JnsShlWlk]{JnsShlWlk} V. F. R. Jones, D. Shlyakhtenko, K. Walker,
\textit{An orthogonal approach to the graded subfactor of a planar
algebra}, {\tt http://math.berkeley.edu/\verb1~1vfr}.

\bibitem[JnsSnd]{JnsSnd} V. F. R. Jones and V. S. Sunder, \textit{Introduction
to subfactors}. London Mathematical Society Lecture Note Series, 234. Cambridge University Press, Cambridge, 1997. xii+162 pp. ISBN: 0-521-58420-5

\bibitem[KdySnd]{KdySnd} Vijay Kodiyalam and V. S. Sunder, \textit{On Jones' planar algebras}, Journal of Knot Theory and its Ramifications, \textbf{Vol. 13, No.2} (2004), 219--247.

\bibitem[PmsPpa]{PmsPpa} M. Pimsner and S. Popa, \textit{Entropy and index for subfactors}, Annales Scientifiques de l'E.N.S., \textbf{Vol. 19, No. 1} (1986), 57--106.

\bibitem[Ppa]{Ppa} S. Popa, \textit{An axiomatization of the lattice of higher relative commutants of a subfactor}, Inventiones Mathematicae 
\textbf{120(3)} (1995), 427--445.


\bibitem[PpaShl]{PpaShl} S. Popa and D. Shlyakhtenko, \textit{Universal properties of $L(F_\infty)$ in subfactor theory}, Acta Mathematica \textbf{191} (2003), 225--257.

\bibitem[Tks]{Tks} M. Takesaki, \textit{Theory of Operator Algebras I},
 Springer-Verlag, New York-Heidelberg, 1979. vii+415 pp. ISBN: 0-387-90391-7.



\end{thebibliography}

\end{document}